\documentclass[review]{elsarticle}
\usepackage{graphicx}
\usepackage{subcaption}
\usepackage{lineno,hyperref}
\usepackage[utf8x]{inputenc}
\usepackage{amsmath}
\usepackage{amsfonts}
\usepackage{amssymb}
\usepackage{cleveref}
\usepackage[dvipsnames]{xcolor}
\usepackage{pgfplots} 
\pgfplotsset{compat=newest} 
\usetikzlibrary{positioning}
\usetikzlibrary{arrows}
\usetikzlibrary{shapes.multipart}
\usepackage{mathrsfs}
\usepackage{tikz}
\usepackage{tikz-cd}
\newtheorem{theorem}{Theorem}
\newtheorem{definition}{Definition}
\newtheorem{property}{Property}

\newtheorem{proof}{Proof}
\newtheorem{lemma}{Lemma}

\newcommand{\lii}[1]{{\mathcal{I}_+^{#1}}}
\newcommand{\rii}[1]{{\mathcal{I}_-^{#1}}}
\newcommand{\ldd}[1]{{\mathcal{D}_+^{#1}}}
\newcommand{\rdd}[1]{{\mathcal{D}_-^{#1}}}
\newcommand{\inner}[2]{\left(#1,#2\right)_\Omega}

\newcommand{\hsss}{\widehat{H}_0^s(\Omega)}
\newcommand{\hilbert}[1]{\widehat{H}_0^{#1}(\Omega)}

\newcommand{\llo}{L^2(\Omega)}
\newcommand{\blinear}[2]{B\left[#1,  #2\right]}
\newcommand{\RN}[1]{%
	\textup{\uppercase\expandafter{\romannumeral#1}}%
}

\DeclareMathOperator\Arg{Arg}

\usepackage{ulem}

\newcommand{\rev}[1]{{\color{blue} #1}}

\newcommand{\li}[1]{{\mathcal{I}_{a+}^{#1}}}
\newcommand{\ri}[1]{{\mathcal{I}_{b-}^{#1}}}
\newcommand{\ld}[1]{{\mathcal{D}_{a+}^{#1}}}
\newcommand{\rd}[1]{{\mathcal{D}_{b-}^{#1}}}
\newcommand{\LL}{L^2(\Omega)}

\journal{Journal for review}

\begin{document}

\begin{frontmatter}

\title{Spectral analysis of a family of nonsymmetric fractional elliptic operators}
\author[aus]{Quanling Deng}
\ead{Quanling.Deng@anu.edu.au}
\author[unr]{Yulong Li\corref{mycorrespondingauthor}}
\ead{yulongl@unr.edu;liyulong0807101@gmail.com}

\address[aus]{School of Computing, Australian National University, Canberra, ACT 2601, Australia.}
\address[unr]{Department of Mathematics and Statistics, University of Nevada, 
	1664 N. Virginia Street, Reno, NV 89557, USA.}

\cortext[mycorrespondingauthor]{Corresponding author}


\begin{abstract}
In this work, we investigate the spectral problem  
\[
\begin{cases}
	\ld{\alpha}\rd{\beta}u=\lambda u,\quad x\in (a,b)\\
	u(a)=u(b)=0,  1<\alpha+\beta<2,
\end{cases}
\]
where the operators $\ld{\alpha}$ and $\rd{\beta}$ are left- and right-sided Riemann-Liouville derivatives. These operators are nonlocal and nonsymmetric, however, share certain classic elliptic properties. The eigenvalues correspond to the roots of a class of certain special functions. Compared with classic Sturm-Liouville problems, the most challenging part is to set up the framework for analyzing these nonlocal operators, which requires developing new tools. We prove the existence of the real eigenvalues, find the range for all possible complex eigenvalues, explore the graphs of eigenfunctions, and show numerical findings on the distribution of eigenvalues on the complex plane.
\end{abstract}

\begin{keyword}
spectral, nonsymmetric, principal eigenvalue, fractional derivative, mixed derivative.
\end{keyword}

\end{frontmatter}

\section{Introduction and motivation}

In this paper, we focus on the spectral analysis of 
\begin{equation}\label{eq:pde}
	\begin{cases}
		Au=\lambda u  \quad\text{in} \quad \Omega=(a, b),\quad \\
		u|_{\partial \Omega}=0, \\
		(Au)(x):=\ld{\alpha}\rd{\beta}u, \quad 0\leq \alpha, \beta \leq 1, \quad 1<\alpha+\beta<2,
	\end{cases}
\end{equation}
where $(a,b)$ is an arbitrary bounded interval.  
Herein, $\ld{\alpha},\rd{\beta}$ represent standard Riemann-Liouville (R-L) derivatives (see \Cref{def:lrintegralderivative} in \ref{appendix-sobolev}).
Problem \eqref{eq:pde} can be challenging when $\alpha\neq \beta$. 
In recent years, attempts have been made to develop the elliptic theory that is analog to the classic elliptic theory for the boundary value problem
\begin{equation}\label{eq:BVP}
	\begin{cases}
		Au=f  \quad\text{in} \quad \Omega=(a, b),\quad \\
		u|_{\partial \Omega}=0.
	\end{cases}
\end{equation}
It has been shown that \Cref{eq:BVP} maintains certain elliptic properties such as the Maximum principle and Hopf's Lemma (see \cite{MR4414583, YLisubmitted} for details).  However, the corresponding counterpart of spectral results are yet to complete. The complexity becomes obvious if one rewrites the operator $A$ equivalently into a form that contains the singular integral of the Hilbert transform. 

In the general case when $f=f(u,x)$, the problem \eqref{eq:BVP} includes various linear and nonlinear situations, which arise in many fields of science and engineering such as in the viscoelasticity theory, optimal control and modeling.
Such problems have been extensively studied in, for example \cite{MR2351655, MR2491124, MR2332922, MR4151111, MR4036338},
for various problem settings with different boundary and initial conditions.
We also refer to \cite{MR3449652, MR3320682, MR4102344, MR4437921} for other variants of the operator. 
Again, among the investigations, one of the major challenges is to analyze the operator $A$ theoretically.

In this paper, we intend to set up a suitable framework for the analysis of operator $A$ and  study the spectral properties of this class of operators both theoretically and numerically.
From the  usual fractional Sturm-Liouville problem
\begin{equation}\label{eq:auxiliary}
	\begin{cases}
		\tilde{A}u=\lambda u  \quad\text{in} \quad \Omega=(a, b),\quad \\
		u|_{\partial \Omega}=0,\\
		(\tilde{A}u)(x):=D^t u, 1<t<2,
	\end{cases}
\end{equation}
 it is well-known that $\lambda\in \mathbb{C}$ is an eigenvalue of \eqref{eq:auxiliary} if and only if $\lambda$ is a zero of $E_{\alpha,2}[-\lambda]$, when $\tilde{A}$ represents  Caputo derivative (\cite{MR0414982}); and if and only if $\lambda$ is a zero of $E_{\alpha,\alpha}[-\lambda]$, when $\tilde{A}$ represents R-L fractional derivative (\cite{MR4414583}). 
Herein, $E_{\cdot, \cdot}[\cdot]$ is a two-parameter Mittag-Leffler function.
With this in mind, one may expect that the eigenvalues of \eqref{eq:pde} correspond to the zeros of a class of generalized special functions, that is, generalizations of Mittag-Leffler functions.
The simplicity of the eigenvalues of  \eqref{eq:pde} is unknown, and it is not even clear 
whether  \eqref{eq:pde} has finite real (or complex) eigenvalues and 
whether the eigenfunctions have finite oscillation periods. 
In this work, we provide partial results to these open questions. 
We mainly prove that all the possible complex values must be located inside the cone $|\Arg \lambda|\leq \frac{|\beta-\alpha|\pi }{2}$ (\Cref{theorem:range}) and problem \eqref{eq:pde} has at least one real eigenvalue (\Cref{the:theorem-principal}). 
Furthermore, based on theoretical analysis, we perform the finite element analysis to study the graphs of eigenfunctions and the distribution of eigenvalues on the complex plane.
\section{Notation}
We collect all the symbols that will appear in the work: Let $\Omega=(a,b)$ represent a finite interval. All the functions considered in this work are assumed to be complex-valued unless otherwise specified. The $C_0^\infty(\Omega)$ consists of all infinitely differentiable functions with compact support in $\Omega$. $(f,g)_\Omega$ denotes the standard inner product $\int_{\Omega}f\overline{g}\, dx$. Riemann-Liouville integrals $\li{t}, \ri{t}, \lii{t}, \rii{t}$ and Riemann-Liouville derivatives $\ld{t}, \rd{t}, \ldd{t}, \rdd{t}$ are defined as in \Cref{def:lrintegralderivative} in \ref{appendix-sobolev}.

\section{Preliminary on fractional Sobolev space}
The fractional Sobolev space $W^{s,2}(\mathbb{R})=\widehat{H}^s(\mathbb{R})$, $s
\geq 0$ is a useful Hilbert space, which can be defined in different ways with equal or equivalent (semi)norms. For elliptic problems, it provides a natural functional space where weak solutions are sought. Furthermore, the true solutions would be recovered by raising the regularity therein. It turns out that the classical definitions of $\widehat{H}^s(\mathbb{R})$ are not quite fit for the analysis of many types of fractional-order differential equations and therefore need to be suitably tailored to fit the analysis. We shall adopt the following characterization via weak fractional derivatives for our purpose.

\begin{definition}[Weak fractional derivative](\cite{glsima18}, Section 3)\label{Equivalent-definition}
	Assume $ s\geq 0$ and  $u\in L^2(\mathbb{R})$. Then  $u \in \widehat{H}^s(\mathbb{R})$ if and only if there is a unique $\psi_1 \in L^2(\mathbb{R})$  such that 
	\begin{equation}\label{eq:1-definition}
	\int_\mathbb{R} u \cdot \ldd{s} \psi\, dx =\int_\mathbb{R} \psi_1 \cdot \psi\, dx,\quad \forall \psi\in C^\infty_0(\mathbb{R}).
	\end{equation}
	Similarly, 
	$u \in \widehat{H}^s(\mathbb{R})$ if and only if there is a unique $\psi_2\in L^2(\mathbb{R})$  such that 
	\begin{equation}\label{eq:2-definition}
	\int_\mathbb{R} u \cdot \rdd{s} \psi\, dx =\int_\mathbb{R} \psi_2 \cdot \psi\, dx \quad \forall \psi\in C^\infty_0(\mathbb{R}).
	\end{equation}

\end{definition}

Herein, $\psi_1, \psi_2$ are called the weak Riemann-Liouville derivatives of $u$, denoted by $ \rdd{s} u, \ldd{s} u$, respecitively. 
They are understood in the weak sense as in \eqref{eq:1-definition} and \eqref{eq:2-definition}. 

It is well-known that $\widehat{H}^s(\mathbb{R})$ can be defined via Fourier transform, i.e.,
\begin{equation}
	\widehat{H}^s(\mathbb{R}) = \left \{ w \in L^2(\mathbb{R}) : \int_{\mathbb{R}} (1 + |2\pi\xi|^{2s}) |\widehat{w}(\xi) |^2 \, {\rm d} \xi < \infty \right \},
\end{equation}
where $\widehat{w}$ is Fourier transform and
\[
\|w\|_{\widehat{H}^s (\mathbb{R})}:=\left(\|w\|^2_{L^2(\mathbb{R})} +|w|^2_{\widehat{H}^s(\mathbb{R})}\right)^{1/2}, ~~\text{with}~~
|w|_{\widehat{H}^s(\mathbb{R})}:=\|(2\pi\xi)^s \widehat{w}\|_{L^2(\mathbb{R})}.
\]

The following property implies that, with the weak R-L derivatives, one can define the norm and seminorm that are equivalent to the ones defined via the Fourier transform.
\begin{property}(\cite{glsima18}, Section 3)\label{weak-equivalence-norm}
	Assume $u\in \widehat{H}^s(\mathbb{R})$, $s\geq0$, then 
	\begin{equation}
	|u|_{\widehat{H}^s(\mathbb{R})} = \|\ldd{s} u\|_{L^2(\mathbb{R})}=\|\rdd{s} u\|_{L^2(\mathbb{R})},
	\end{equation}
	where $\rdd{s} u$ and $\ldd{s} u$ are understood in the weak sense.
\end{property}

We use in the analysis below these equivalent seminorms and norms to analyze equations involving Riemann-Liouville derivative operators.

Since  $C_0^\infty(\Omega)$ is dense in $\widehat{H}^s(\mathbb{R})$, we can define the following analog by restricting to the bounded domain. This is the main functional space that we are going to adopt in this work.
\begin{definition}\label{def:ForInterval}
	Given $s\geq 0$.
	\begin{equation}
	\widehat{H}^s_0(\Omega):=\{\text{Closure of}\, \, u\in C_0^\infty(\Omega) \, \text{with respect to norm}\, \|\tilde{u}\|_{\widehat{H}^s(\mathbb{R})} \},
	\end{equation}
	where  $\tilde{u}$ represents the trivial continuation of $u(x)$ by $0$ beyond $\Omega$. It is endowed with semi-norm and norm
	\[
	|u|_{\widehat{H}^s_0(\Omega)}:=|\tilde{u}|_{\widehat{H}^s(\mathbb{R})},
	\|u\|_{\widehat{H}^s_0(\Omega)}:=\|\tilde{u}\|_{\widehat{H}^s(\mathbb{R})}.
	\]
\end{definition}
As a subspace of $\widehat{H}^s(\mathbb{R})$, also a  Hilbert space, $\hsss$ can be characterized by the following.
\begin{property}[Theorem 1.4.2.2, p. 24, \cite{MR775683}]\label{pro:subspace}
	\[
	\hsss=\{u\in \widehat{H}^s(\mathbb{R}): u\equiv 0 \quad \text{on}\quad \mathbb{R}\backslash \Omega\}.
	\]
\end{property}
%
%
%
With \Cref{Equivalent-definition} and \Cref{def:ForInterval}, we can characterize the integral representation for elements of $\hilbert{s}$, which will be utilized in our subsequent analysis.
\begin{property}\label{pro:alternate-form}
	Given $0<s<1$. $u\in \widehat{H}^s_0(\Omega)$ can be represented as 
	\begin{equation}\label{eq:representation-left-right}
	u(x)=\ri{s}\psi_1=\li{s}\psi_2,\quad x\in \Omega
	\end{equation}
	for certain   $\psi_1$, $\psi_2 \in L^2(\Omega)$. As a consequence, the R-L fractional derivatives of elements of $\hsss$, $\ld{s}u$ and $\rd{s}u$,  exist a.e. and coincide with $\psi_1$, $\psi_2$, respectively.
\end{property}

\begin{proof}
 For the case $1/2<s<1$, \eqref{eq:representation-left-right} has been justified in previous work  \cite{yulongginting}. In the same spirit, we justify \Cref{pro:alternate-form} for the case $0<s\leq 1/2$ only.
 
Since $u\in \hilbert{s}$, according to \Cref{pro:subspace}, its extension $\widetilde{u}$ belongs to $\widehat{H}^s(\mathbb{R})$. By \Cref{Equivalent-definition}, there exist $\psi_1,\psi_2\in L^2(\mathbb{R})$ such that
	 \begin{equation}\label{eq:deduceddef}
	 	\int_\mathbb{R} \widetilde{ u}\cdot \ldd{s}v\, dx=\int_\mathbb{R} \psi_1 \cdot v \, dx \quad \text{and}\quad \int_\mathbb{R} \widetilde{ u}\cdot \rdd{s}v\, dx=\int_\mathbb{R} \psi_2\cdot v \, dx,\quad \forall v\in C^\infty_0(\mathbb{R}).
	 \end{equation}
	 Since $\widetilde{u} = u$ in $\Omega$ and $\widetilde{u}=0$ in $\mathbb{R}\backslash\Omega$, \eqref{eq:deduceddef} yields
	 \begin{equation}
	 	\int_\Omega u\cdot \ld{s}v\, dx=\int_\Omega \psi_1 \cdot v \, dx\quad \text{and}\quad \int_\Omega u\cdot \rd{s}v\, dx=\int_\Omega\psi_2\cdot v\, dx,\quad \forall v\in C^\infty_0(\Omega).
	 \end{equation}
	 Interchanging the differentiation and integration gives
	 \[
	 \int_\Omega u\cdot \li{1-s}Dv\, dx=\int_\Omega \psi_1 \cdot v \, dx ~\text{and}~ -\int_\Omega u\cdot \ri{1-s}Dv\, dx=\int_\Omega\psi_2\cdot v\, dx,~\forall v\in C^\infty_0(\Omega).
	 \]
	 
	 We use  fractional integration by parts (eq. (2.20), p. 34, \cite{MR1347689}) to arrive at
	 \begin{equation*}
	 	\int_\Omega (\ri{1-s}u- \ri{1}\psi_1 )Dv\, dx=0 ~\text{and}~ \int_\Omega(- \li{1-s}u+\li{1}\psi_2) Dv\, dx=0,~\forall v\in C^\infty_0(\Omega).
	 \end{equation*}
	 Since $v$ is arbitrary, this further implies  that (Lemma 8.1, p. 204, \cite{MR2759829})
	 \begin{equation*}
	 	\ri{1-s}u=\ri{1}\psi_1+c_1 \quad \text{and} \quad \li{1-s}u=\li{1}\psi_2+c_2, \, \text{ a.e. in } \Omega, 
	 \end{equation*}
	 for certain constants $c_1$ and $c_2$. Taking the left and right Riemann-Liouville derivatives on both sides of these last equations gives
	 \begin{equation*}
	 	u(x)=\ri{s}\psi_1+c_1(b-x)^{-(1-s)}\quad \text{and }\quad u(x)=\li{s}\psi_2+c_2(x-a)^{-(1-s)}, \, \text{a.e. in } \Omega.
	 \end{equation*}

Notice that $0<s\leq 1/2$ and $u\in \LL$, hence, $c_1$ and $c_2$ have to be zero. It follows that
	 \begin{equation}\label{eq:constantsgone}
	 	u(x)=\ri{s}\psi_1=\li{s}\psi_2, \,a.e.\quad x\in \Omega,
	 \end{equation}
	 which is the desired result.
\end{proof}

We now establish the following semi-norm equivalence relationship.
\begin{lemma}\label{lemma:bound-1}
	Suppose $0<t<1$, $t\neq 1/2$ and $u\in \hilbert{t}$. Let $\tilde{u}$ denote the extension of $u$ by zero outside $\Omega$, then
	\begin{equation}\label{eq:bound-1}
			|\cos(t\pi)|	\|\ld{t}u\|_{\LL}\leq 	\|\rd{t}u\|_{\LL}\leq  \|\rdd{t}\tilde{u}\|_{L^2(\mathbb{R})} \leq \frac{1}{|\cos(t\pi)|} 	\|\ld{t}u\|_{\LL}
	\end{equation}
and
	\begin{equation}\label{eq:bound-1-2}
	|\cos(t\pi)|	\|\rd{t}u\|_{\LL}\leq 	\|\ld{t}u\|_{\LL}\leq  \|\ldd{t}\tilde{u}\|_{L^2(\mathbb{R})} \leq \frac{1}{|\cos(t\pi)|} 	\|\rd{t}u\|_{\LL}.
	\end{equation}
\end{lemma}
\begin{proof}
To prove \eqref{eq:bound-1} and \eqref{eq:bound-1-2} for complex-valued functions, it suffices to show that \Cref{lemma:bound-1} is valid for all real-valued function $u$. Now suppose that $u$ is an arbitrary real-valued function and $u\in \hilbert{t}$.  Then the identity 
	\begin{equation}\label{eq:wellknwon}
\cos(t\pi)\int_{\mathbb{R}}|\ldd{t}\tilde{u}|^2 \,dx=\cos(t\pi)	\int_{\mathbb{R}}|\rdd{t}\tilde{u}|^2 \,dx=\int_{\Omega} \ld{t}u \cdot \rd{t}u \, dx
	\end{equation}
	holds (\cite{glsima18}, Theorem 4.1).
Using H\"older's inequality and Peter–Paul inequality, we arrive at
	\begin{equation}\label{eq:young}
		\begin{aligned}
		\big|	\int_{\Omega} \ld{t}u\, \rd{t}u \, dx\big |&\leq \frac{1}{4\epsilon}\|\ld{t}u\|_{\LL}^2+\epsilon\|\rd{t}u\|_{\LL}^2\\
			&\leq  \frac{1}{4\epsilon}\|\ld{t}u\|_{\LL}^2+\epsilon\|\rdd{t}\tilde{u}\|_{L^2(\mathbb{R})}^2.
		\end{aligned}
	\end{equation}
By choosing $\epsilon=\frac{|\cos(t\pi)|}{2}$ and applying \eqref{eq:wellknwon} to the left-hand side, \eqref{eq:young} yields
\[
\int_{\mathbb{R}}|\rdd{t}\tilde{u}|^2 \,dx\leq \frac{1}{\cos^2(t\pi)} \|\ld{t}u\|_{\LL}^2.
\]
Hence,
\begin{equation}\label{eq:lf}
\|\rd{t}u\|_{\LL}\leq \|\rdd{t}\tilde{u}\|_{L^2(\mathbb{R})}\leq \frac{1}{|\cos(t\pi)|} \|\ld{t}u\|_{\LL}.
\end{equation}
Likewise, we can directly verify that
\begin{equation}\label{eq:rt}
	\|\ld{t}u\|_{\LL}\leq \frac{1}{|\cos(t\pi)|} \|\rd{t}u\|_{\LL}.
\end{equation}
\eqref{eq:lf} together with \eqref{eq:rt} yields \eqref{eq:bound-1}. \eqref{eq:bound-1-2} can be proved in a similar fashion. 
This completes the proof.
\end{proof}

We also need the following auxiliary lemma, which is subtle but useful for computing the norms of operators with mixed R-L derivatives later.
\begin{lemma}\label{lem:subtle-seminorm}
	Suppose $u\in \hilbert{t}, 0<t<1$. Let $0<\sigma<\min\{1/2, t\}$. Then  $\ld{t-\sigma}u, \rd{t-\sigma}u\subset \hilbert{\sigma}$.
\end{lemma}
\begin{proof}
	Suppose $u\in \hilbert{t}$. By \Cref{pro:alternate-form}, $u$ admits integral representations
	\[
	u=\li{t}J_1=\ri{t}J_2,\quad \text{for certain}\quad J_1, J_2\in \LL.
	\]
	Hence, 
	\[
	\ld{t-\sigma}u=\li{\sigma}J_1,\quad \rd{t-\sigma}u=\ri{\sigma}J_2.
	\]
	Then, due to $0<\sigma<1/2$, the extensions of $\ld{t-\sigma}u$ and $\rd{t-\sigma}u$, i.e.
	\[
	\begin{split}
	\widetilde{	\ld{t-\sigma}u}=
	\begin{cases}
			\ld{t-\sigma}u, \quad &x\in \Omega\\
			0, &x\in \mathbb{R}\backslash \Omega
	\end{cases},
\quad
	\widetilde{	\rd{t-\sigma}u}=
\begin{cases}
	\rd{t-\sigma}u, \quad &x\in \Omega\\
	0, &x\in \mathbb{R}\backslash \Omega
\end{cases}
\end{split}
	\]
also admit integral representations (Theorem 13.10, p. 236, \cite{MR1347689})
	\[
	\widetilde{	\ld{t-\sigma}u}=\lii{\sigma}J_1^*,\quad	\widetilde{	\rd{t-\sigma}u}=\rii{\sigma}J_2^*, x\in \mathbb{R},\quad \text{for certain}\quad J_1^*, J_2^*\in L^2(\mathbb{R}).
	\]
This implies that  $\widetilde{	\ld{t-\sigma}u}$ and $\widetilde{	\rd{t-\sigma}u}$ satisfy \Cref{Equivalent-definition}, respectively, i.e.,
\[
	\int_\mathbb{R} \widetilde{	\ld{t-\sigma}u} \cdot \rdd{\sigma} \psi\, dx =\int_\mathbb{R} J_1^* \cdot \psi\, dx \quad \forall \psi\in C^\infty_0(\mathbb{R})
\]
and
\[
	\int_\mathbb{R} \widetilde{	\rd{t-\sigma}u} \cdot \ldd{\sigma} \psi\, dx =\int_\mathbb{R} J_2^* \cdot \psi\, dx,\quad \forall \psi\in C^\infty_0(\mathbb{R})
\]
by fractional integration by parts (eq. (5.17), p. 96, \cite{MR1347689}). Therefore, by definition, we obtain
\[
\widetilde{	\ld{t-\sigma}u},\widetilde{	\rd{t-\sigma}u}\in \widehat{H}^\sigma(\mathbb{R}).
\]
It follows that
\[
\ld{t-\sigma}u,\rd{t-\sigma}u\in \hilbert{\sigma}
\]
by \Cref{pro:subspace}, which completes the proof.
\end{proof}
Now we are ready to present the following property involving mixed fractional derivatives.
\begin{lemma}\label{lemma:Bound}
	Suppose $t_1>0, t_2>0$, $0<t_1+t_2<1$, $0<t_1<1/2$ and $u\in \hilbert{t_1+t_2}$, then
	\begin{equation}\label{eq:bound-2}
	\cos(t_1\pi)	\|\ld{t_1+t_2}u\|_{\LL}\leq 	\|\rd{t_1}\ld{t_2}u\|_{\LL}\leq  \frac{1}{\cos(t_1\pi)}	\|\ld{t_1+t_2}u\|_{\LL} .
	\end{equation}
Analogously,
	\begin{equation}\label{eq:bound-3}
	\cos(t_1\pi)	\|\rd{t_1+t_2}u\|_{\LL}\leq 	\|\ld{t_1}\rd{t_2}u\|_{\LL}\leq  \frac{1}{\cos(t_1\pi)}	\|\rd{t_1+t_2}u\|_{\LL} .
\end{equation}
\end{lemma}
\begin{proof}
Since $u\in \hilbert{t_1+t_2}$ and $0<t_1<1/2$,  \Cref{lem:subtle-seminorm} is applicable and we know that $\ld{t_2}u \in \hilbert{t_1}$. Then, applying \Cref{lemma:bound-1} gives
\[
\cos(t_1\pi) \|\ld{t_1}\ld{t_2}u\|_{\LL}    \leq  	\|\rd{t_1}\ld{t_2}u\|_{\LL}\leq \frac{1}{\cos(t_1\pi)} \|\ld{t_1}\ld{t_2}u\|_{\LL}.
\]
Notice the semigroup property of fractional derivatives is valid to the leftmost and rightmost sides due to  $u\in \hilbert{t_1+t_2}$. This is
\[
\cos(t_1\pi) \|\ld{t_1+t_2}u\|_{\LL}    \leq  	\|\rd{t_1}\ld{t_2}u\|_{\LL}\leq \frac{1}{\cos(t_1\pi)} \|\ld{t_1+t_2}u\|_{\LL}.
\]
\eqref{eq:bound-3} can be analogously proved.
\end{proof}

\section{Variation formulation, weak and true solutions}\label{sec:variation}
Let $s=\frac{\alpha+\beta}{2}$ throughout \Cref{sec:variation}. We first construct the following sesquilinear form
 $\blinear{u}{v}$  for $u,v\in\hsss$.  Although \eqref{eq:pde} can be readily solved into a generalized fractional integral equation involving $u$ and $\lambda$, without using the sesquilinear form,  we intend to provide a foundation for later numerical algorithm and experiments for computing the eigenvalues and plotting the graphs of eigenfunctions.
\begin{definition}
	\begin{enumerate}
		\item The sesquilinear form $\blinear{\cdot}{\cdot}$ associated with the fractional operator $A$ in \eqref{eq:pde} is defined as
		\begin{equation}\label{eq:bilinear}
		 \begin{split}
		 	 \blinear{u}{v}:=
		\begin{cases}
			(\rd{\frac{\alpha+\beta}{2}}u, \ld{\frac{\beta-\alpha}{2}}\rd{\alpha}v)_\Omega,\quad & \text{if}\quad \beta>\alpha,\\
			(\rd{\alpha}u, \rd{\alpha}v)_\Omega, &\text{if}\quad  \beta=\alpha,\\
			(\ld{\frac{\alpha-\beta}{2}}\rd{\beta}u, \rd{\frac{\alpha+\beta}{2}}v)_\Omega,
					& \text{if}\quad \beta<\alpha,			
		\end{cases}
	 \end{split}
		 \end{equation}
		for $u,v\in\hsss$. Here the symbol $(f,g)_\Omega$ denotes $\int_{\Omega}f\overline{g}\, dx$
		\item We say that $u\in\hsss$ is a weak solution of the boundary-value problem \eqref{eq:pde} if
		\[ \blinear{u}{v}=\inner{f}{v} \]
		for all $v\in\hsss$.
	\end{enumerate}
\end{definition}

Herein, \eqref{eq:bilinear} is well-defined for each case. Namely, each component in the inner product belongs to $\LL$: when $\beta> \alpha$, $\rd{\frac{\alpha+\beta}{2}}u\in \LL$ by  \Cref{pro:alternate-form} and $\ld{\frac{\beta-\alpha}{2}}\rd{\alpha}v\in \LL$ can be directly verified by noting that it can be written as $\ld{\frac{\beta-\alpha}{2}}\rd{\alpha}v=-\li{1-\frac{\beta-\alpha}{2}}\ri{1-\alpha}v''-\rd{\alpha}v\big |_{x=a}\cdot \ld{\frac{\beta-\alpha}{2}}1$; when $\beta=\alpha$, by $\rd{\alpha}u\in \LL$ by  \Cref{pro:alternate-form} and $\rd{\alpha}v\in \LL$ is straightforward; when $\beta<\alpha$, $\ld{\frac{\alpha-\beta}{2}}\rd{\beta}u\in \LL$ by \Cref{lem:subtle-seminorm} and $\rd{\frac{\alpha+\beta}{2}}v\in \LL$ follows naturally.
\begin{lemma}\label{lemma:lemma1}
	$\blinear{\cdot}{\cdot}$ is bounded and strongly coercive on $\hsss$. That is, there are real constants $c_1,c_2>0$ such that
	\begin{equation}\label{eq:B-1}
		\left|\blinear{u}{v}\right|\leq c_1  \left|\left|u\right|\right|_{\hsss} \left|\left|v\right|\right|_{\hsss}
	\end{equation}
	and
	\begin{equation}\label{eq:B-2} 
	\left|\blinear{u}{u}\right|\geq c_2 \left|\left|u\right|\right|_{\hsss}^2 
	\end{equation}
	for all $u,v\in\hsss$.
\end{lemma}

\begin{proof}
1.  For $\forall$ $u,v\in\hsss$,   applying  the H\"{o}lder inequality to \eqref{eq:bilinear}, we have 
		\[
			\begin{split}
				\left|\blinear{u}{v}\right|	\leq 	\begin{cases}
					\|\rd{\frac{\alpha+\beta}{2}}u\|_{\LL}\|\ld{\frac{\beta-\alpha}{2}}\rd{\alpha}v\|_{\LL},\quad & \text{if}\quad \beta>\alpha,\\
					\|\rd{\alpha}u\|_{\LL} \|\rd{\alpha}v\|_{\LL}, &\text{if}\quad  \beta=\alpha,\\
					\|\ld{\frac{\alpha-\beta}{2}}\rd{\beta}u\|_{\LL} \|\rd{\frac{\alpha+\beta}{2}}v\|_{\LL},
					&\text{if}\quad \beta<\alpha.		
				\end{cases}
			\end{split}
		\]
	Since $0<|\frac{\beta-\alpha}{2}|<1/2$, we apply \Cref{lemma:Bound} to both cases of $\beta>\alpha$ and $\beta<\alpha$ above. We obtain
		\begin{displaymath}
				\begin{split}
			\left|\blinear{u}{v}\right|&	\leq 
				\begin{cases}
			\cos^{-1}(\frac{\beta-\alpha}{2}\pi)	\|\rd{s}u\|_{\LL}\|\rd{s}v\|_{\LL},\quad &\text{if}\quad \beta>\alpha,\\
				\|\rd{s}u\|_{\LL} \|\rd{s}v\|_{\LL}, &\text{if}\quad  \beta=\alpha,\\
			\cos^{-1}(\frac{\alpha-\beta}{2}\pi)	\|\rd{s}u\|_{\LL} \|\rd{s}v\|_{\LL},
				& \text{if}\quad \beta<\alpha,			
			\end{cases}\\
			&\leq 	\begin{cases}
				\cos^{-1}(\frac{\beta-\alpha}{2}\pi)	\|u\|_{\hilbert{s}}\|v\|_{\hilbert{s}},\quad & \text{if}\quad \beta>\alpha,\\
				\|u\|_{\hilbert{s}}\|v\|_{\hilbert{s}}, &\text{if}\quad  \beta=\alpha,\\
				\cos^{-1}(\frac{\alpha-\beta}{2}\pi)	\|u\|_{\hilbert{s}} \|v\|_{\hilbert{s}},
				&\text{if}\quad \beta<\alpha,			
			\end{cases}\\
			&\leq c_1||u||_{\hsss} ||v||_{\hsss},
		\end{split}
		\end{displaymath}
		where $c_1=\max \{\cos^{-1}(\frac{|\beta-\alpha|}{2}), 1\}$. We have proved \eqref{eq:B-1}.

2.  To show \eqref{eq:B-2}, pick any $u\in \hilbert{s}$. Let us consider $\tilde{u}$, the trivial zero extension of $u$ from $\Omega$ to $\mathbb{R}$, and write it into $\tilde{u}(x)=f_1(x)+if_2(x), x\in \mathbb{R}$, where $f_1,f_2$ are real-valued functions (Note since $\tilde{u}$ has compact support in $\Omega$, so are $f_1, f_2$). Then
		\begin{equation}\label{eq:goalequ}
		\blinear{u}{u}=\blinear{f_1}{f_1}+\blinear{f_2}{f_2}+i(\blinear{f_2}{f_1}-\blinear{f_1}{f_2}).
		\end{equation}
			
 In the following, we first calculate 	$\blinear{f_1}{f_1}$. Due to  $f_1,f_2\in \hilbert{s}$, we directly obtain
	\[
	\begin{split}
	\blinear{f_1}{f_1}&=
			\begin{cases}
		(\rd{\frac{\alpha+\beta}{2}}f_1, \ld{\frac{\beta-\alpha}{2}}\rd{\alpha}f_1)_\Omega,\quad &\text{if}\quad \beta>\alpha,\\
		(\rd{\alpha}f_1, \rd{\alpha}f_1)_\Omega, &\text{if}\quad  \beta=\alpha,\\
		(\ld{\frac{\alpha-\beta}{2}}\rd{\beta}f_1, \rd{\frac{\alpha+\beta}{2}}f_1)_\Omega,
		& \text{if}\quad \beta<\alpha,			
	\end{cases}\\
&=\begin{cases}
	(\rd{\frac{\beta-\alpha}{2}}\rd{\alpha}f_1, \ld{\frac{\beta-\alpha}{2}}\rd{\alpha}f_1)_\Omega,\quad &\text{if}\quad \beta>\alpha,\\
	(\rd{\alpha}f_1, \rd{\alpha}f_1)_\Omega, &\text{if}\quad  \beta=\alpha,\\
	(\ld{\frac{\alpha-\beta}{2}}\rd{\beta}f_1, \rd{\frac{\alpha-\beta}{2}}\rd{\beta}f_1)_\Omega,
	& \text{if}\quad \beta<\alpha.		
\end{cases}
\end{split}
\]
Note $0<\frac{|\alpha-\beta|}{2}<1/2$, which allows us to  invoke \Cref{lem:subtle-seminorm} to obtain that $\rd{\alpha}f_1\in \hilbert{\frac{\beta-\alpha}{2}}$ when $\beta>\alpha$, and that $\ld{\beta}f_1\in \hilbert{\frac{\alpha-\beta}{2}}$ when $\alpha>\beta$. Applying  identity \eqref{eq:wellknwon} for these two different cases, we arrive at
\[
\begin{split}
&	=
	\begin{cases}
	\cos(\frac{\beta-\alpha}{2}\pi)(\rdd{\frac{\beta-\alpha}{2}}\widetilde{	\rd{\alpha}f_1}, \rdd{\frac{\beta-\alpha}{2}}\widetilde{	\rd{\alpha}f_1})_\mathbb{R},\quad &\text{if}\quad \beta>\alpha,\\
	(\rd{\alpha}f_1, \rd{\alpha}f_1)_\Omega, &\text{if}\quad  \beta=\alpha,\\
\cos(\frac{\alpha-\beta}{2}\pi)	(\rdd{\frac{\alpha-\beta}{2}}\widetilde{\rd{\beta}f_1}, \rdd{\frac{\alpha-\beta}{2}}\widetilde{\rd{\beta}f_1})_\mathbb{R},
	& \text{if}\quad \beta<\alpha.	
\end{cases}\\
&\geq 
\begin{cases}
	\cos(\frac{\beta-\alpha}{2}\pi)(\rd{s}f_1, \rd{s}f_1)_\Omega,\quad &\text{if}\quad \beta>\alpha,\\
	(\rd{s}f_1, \rd{s}f_1)_\Omega, &\text{if}\quad  \beta=\alpha,\\
	\cos(\frac{\alpha-\beta}{2}\pi)	(\rd{s}f_1, \rd{s}f_1)_\Omega,
	& \text{if}\quad \beta<\alpha.	
\end{cases}
\end{split}
	\]
	By applying \Cref{lemma:bound-1}, this further becomes
	\begin{equation}\label{eq:laststep-inequality}
	\begin{split}
	&\geq 
		\begin{cases}
			\cos(\frac{\beta-\alpha}{2}\pi)|\cos(\frac{\beta+\alpha}{2}\pi)|  \cdot |f_1|_{\hilbert{s}}^2,\quad &\text{if}\quad \beta>\alpha,\\
			|\cos(\frac{\beta+\alpha}{2}\pi)|  \cdot |f_1|_{\hilbert{s}}^2, &\text{if}\quad  \beta=\alpha,\\
			\cos(\frac{\alpha-\beta}{2}\pi)|\cos(\frac{\beta+\alpha}{2}\pi)|\cdot   |f_1|_{\hilbert{s}}^2,
			& \text{if}\quad \beta<\alpha,
		\end{cases}\\
&	\geq 	|\cos(s\pi)| \cdot  |f_1|_{\hilbert{s}}^2.
	\end{split}
	\end{equation}
On the other hand, by the Fractional  Poincar\'e-Friedrichs inequality (\Cref{poincareinequality} in \ref{appendix-sobolev})
\[
	\|g\|_{\LL}\leq \frac{(b-a)^t}{t\Gamma(t)} |g|_{\hilbert{t}}, \forall g\in \hilbert{t}, t>0,
\]	
we know that
\[
|f_1|^2_{\hilbert{s}} \geq \left(1+\frac{(b-a)^{2s}}{s^2\Gamma^2(s)}\right)^{-1} \|f_1\|^2_{\hilbert{s}}.
\]
Applying back substitution into the last line of \eqref{eq:laststep-inequality}, we finally have
\[
\blinear{f_1}{f_1}\geq |\cos(s\pi)|\left(1+\frac{(b-a)^{2s}}{s^2\Gamma^2(s)}\right)^{-1} \|f_1\|^2_{\hilbert{s}}.
\]

3. Similarly, we can verify that
\[
\blinear{f_2}{f_2}\geq |\cos(s\pi)|\left(1+\frac{(b-a)^{2s}}{s^2\Gamma^2(s)}\right)^{-1} \|f_2\|^2_{\hilbert{s}}.
\]
Hence, recalling \eqref{eq:goalequ}, we obtain
\[
\big |\blinear{u}{u}\big| \geq \blinear{f_1}{f_1}+\blinear{f_2}{f_2}\geq |\cos(s\pi)|\left(1+\frac{(b-a)^{2s}}{s^2\Gamma^2(s)}\right)^{-1} \|u\|^2_{\hilbert{s}}.
\]
This gives \eqref{eq:B-2}  by setting $c_2=|\cos(s\pi)|\left(1+\frac{(b-a)^{2s}}{s^2\Gamma^2(s)}\right)^{-1}$, which completes the proof.
\end{proof}

Once we have the boundedness and coercivity, by using the complex variants of the Lax-Milgram Theorem and the regularity of weak solution, we have the following.
\begin{lemma}\label{lemma:lemma2}
	Given any $f\in \llo$, there is a unique weak solution $u\in\hsss$ to the boundary-value problem
	\begin{equation}\label{eq:weaksolutionforf}
		\begin{cases}
			Au=f \quad\text{in} \quad \Omega,\\
			u(a)=u(b)=0.
		\end{cases}
	\end{equation}
And $u$ is also the  true solution (namely, $Au=f$ a.e.  with $u(a)=u(b)=0$) that can be represented as
	\begin{equation}\label{eq:solution}
	u(x)=\ri{\beta}\,(\li{\alpha}f+c\,\ld{1-\alpha}\,1), x\in \Omega
\end{equation}
where 
\begin{equation}\label{eq:constant}
	c=-\frac{\Gamma(\alpha)\Gamma(\beta)(\alpha+\beta-1)}{(b-a)^{\alpha+\beta-1}}\cdot \ri{\beta}\li{\alpha}f \big |_{x=a}.
\end{equation}
Furthermore, we have the norm estimate
	\begin{equation}\label{eq:estimate}
	\|u\|_{\hsss} \leq C_{\alpha,\beta}\|f\|_{L^2(\Omega)}
	\end{equation}
	with a certain positive constant $C_{\alpha,\beta}$  independent of $f$.
\end{lemma}

\begin{proof}
1. Fix a function $f\in \llo$ and define $\left<f,g\right>:=\int_{\Omega}f\overline{g}\, dx$ for any $g\in \LL$. This is a bounded linear functional on $\llo$, and thus on $\hsss$. From \Cref{lemma:lemma1}, $\blinear{\cdot}{\cdot}$ is bounded and coercive on $\hsss$. We apply the Lax-Milgram Theorem to prove the existence and uniqueness of a function $u\in\hsss$ satisfying
		\begin{equation} \label{weak-solution}
			\blinear{u}{v}=\left<f,v\right> 
		\end{equation}
		for all $v\in\hilbert{s}$; $u$ is consequently the unique weak solution of \eqref{eq:weaksolutionforf} by definition.

2. Now, we need to show $Au=f$ a.e. from $\blinear{u}{v}=\left<f,v\right> $, $\forall v\in\hsss$.

Since $u\in \hsss$, $u\in \hilbert{t}$ for any $0\leq t\leq s$ and the R-L derivatives $\ld{t}u, \rd{t}u$ exist and belong to $\LL$. For any fixed $t$, by \Cref{def:ForInterval}, there exists a corresponding sequence $\{u_n\}_{n=1}^\infty\subset C_0^\infty(\Omega)$ such that
\begin{equation}\label{eq:normconvergence}
\lim_{n\rightarrow \infty} u_n\longrightarrow u, \, \lim_{n\rightarrow \infty} \ld{t}u_n\longrightarrow \ld{t}u,\,\text{and}\,\lim_{n\rightarrow \infty} \rd{t}u_n\longrightarrow \rd{t}u\,\,\text{in}\,\,\LL.
\end{equation}
The fact \eqref{eq:normconvergence} allows  \eqref{weak-solution} to be written into the limit form
\begin{equation}\label{eq:weaklinearform}
\blinear{u}{v}=\lim_{n\rightarrow \infty}\blinear{u_n}{v}=(f, v)_\Omega,\quad \forall v\in C_0^\infty(\Omega),
\end{equation}
by observing that
		\begin{equation}
	\begin{split}
\blinear{u}{v}=
		\begin{cases}
			(\rd{\frac{\alpha+\beta}{2}}u, \ld{\frac{\beta-\alpha}{2}}\rd{\alpha}v)_\Omega, &  \beta>\alpha,\\
			(\rd{\alpha}u, \rd{\alpha}v)_\Omega, & \beta=\alpha,\\
			(\ld{\frac{\alpha-\beta}{2}}\rd{\beta}u, \rd{\frac{\alpha+\beta}{2}}v)_\Omega,
			& \beta<\alpha,			
		\end{cases}
	=
			\begin{cases}
		(\rd{\frac{\alpha+\beta}{2}}u, \ld{\frac{\beta-\alpha}{2}}\rd{\alpha}v)_\Omega,& \beta>\alpha,\\
		(\rd{\alpha}u, \rd{\alpha}v)_\Omega, & \beta=\alpha,\\
		(\rd{\beta}u, \rd{\alpha}v)_\Omega,
		&  \beta<\alpha.			
	\end{cases}
	\end{split}
\end{equation} 

3. In the following, we only discuss one of the cases, $\beta>\alpha$, and other cases can be justified similarly. Now suppose $\beta>\alpha$, then
\[
\lim_{n\rightarrow \infty}\blinear{u_n}{v}=\lim_{n\rightarrow \infty}	(\rd{\frac{\alpha+\beta}{2}}u_n, \ld{\frac{\beta-\alpha}{2}}\rd{\alpha}v)_\Omega.
\]
Since $u_n, v\in C_0^\infty(\Omega)$, we use the fractional integration by parts and differentiation by parts, respectively, to reduce it to
\[
-\lim_{n\rightarrow \infty}(\li{1-\alpha}\rd{\beta}u_n, Dv)_\Omega.
\]
We add a suitable constant $c_n=\ri{1-\beta}u_n\big|_{x=a}$ such that we can exchange the differentiation and integration operators
\[
\begin{aligned}
&\lim_{n\rightarrow \infty}(\li{1-\alpha}D(\ri{1-\beta}u_n-c_n), Dv)_\Omega\\
&=\lim_{n\rightarrow \infty}(D\li{1-\alpha}(\ri{1-\beta}u_n-c_n), Dv)_\Omega\\
&=-\lim_{n\rightarrow \infty}(\li{1-\alpha}(\ri{1-\beta}u_n-c_n), v'')_\Omega.
\end{aligned}
\]
Now, it is easy to verify that \eqref{eq:normconvergence} guarantees the convergence of the first component, that is
\begin{equation}\label{eq:equality-1}
-(\li{1-\alpha}(\ri{1-\beta}u-\tilde{c}), v'')_\Omega,
\end{equation}
with $\tilde{c}=\lim_{n\rightarrow \infty}c_n$.
On the other hand, notice
\begin{equation}\label{eq:equality-2}
	(f,v)_\Omega=(\li{2}f, v'')_\Omega.
\end{equation}
Taking into account \eqref{eq:equality-1} and \eqref{eq:equality-2}, \eqref{eq:weaklinearform} now produces
\begin{equation}\label{variationforluma}
(\li{1-\alpha}(\ri{1-\beta}u-\tilde{c})+\li{2}f, v'')_\Omega=0,\quad \forall v\in C_0^\infty(\Omega).
\end{equation}
Hence, \eqref{variationforluma} yields
\begin{equation}\label{eq:almostthere}
\li{1-\alpha}(\ri{1-\beta}u-\tilde{c})+\li{2}f=\tilde{c}_1 x+\tilde{c}_2,\quad\text{for  certain constants}\quad \tilde{c}_1, \tilde{c}_2\in \mathbb{C}.
\end{equation}

4.  Utilizing the boundary condition  $u(a)=u(b)=0$  to simplify and solve \eqref{eq:almostthere} for $u$ with a bit effort, we have
 \[
 u(x)=\ri{\beta}\,(\li{\alpha}f+c\,\ld{1-\alpha}\,1), x\in \Omega
 \]
 with $c$ satisfying \eqref{eq:constant}. Therefore, we have obtained the desired results
  \eqref{eq:solution} and \eqref{eq:constant}. It is a true solution, i.e., $Au=f$ a.e.

5. Lastly, for the estimate~\eqref{eq:estimate}, observe from the coercivity \eqref{eq:B-2}  that 
		\[
	c_2	\|u\|_{\hsss}^2\leq |\blinear{u}{u}|=|\inner{f}{u}|\leq \|f\|_{L^2(\Omega)}\|u\|_{L^2(\Omega)}\leq \|f\|_{L^2(\Omega)}\|u\|_{\hsss},
		\]
		then \eqref{eq:estimate} follows immediately. The proof is complete.
\end{proof}
\section{The distribution of eigenvalues}
According to \Cref{lemma:lemma2}, $u$ is a weak solution if and only if it is a true solution, hence, there is no need to distinguish the weak and true solution in the following. We specify the definition of eigenfunctions as follows.
\begin{definition}
	We say the complex-valued function $u$ is an eigenfunction of \eqref{eq:pde} if  $u$ is a nonzero weak solution of \eqref{eq:pde}, and $\lambda\in \mathbb{C}$ is called the eigenvalue associated with $u$.
\end{definition}
The following theorem characterizes the distribution of all possible eigenvalues.
\begin{theorem}\label{theorem:range}
	Suppose $\lambda\in \mathbb{C}$ is an eigenvalue of \eqref{eq:pde}, then  $|\Arg \lambda|\leq \frac{|\beta-\alpha|\pi }{2}$.
\end{theorem}
\begin{proof}
	1. Suppose $u$ is an  eigenfunction corresponding to eigenvalue $\lambda$, then
	\[
	\blinear{u}{u}=\lambda \|u\|^2_\Omega.
	\]
	Since $u$ is nonzero, we observe
	\begin{equation}\label{eq:argument}
	\Arg \lambda=\Arg \frac{	\blinear{u}{u}}{\|u\|^2_\Omega}=\Arg \blinear{u}{u}.
	\end{equation}
	
	2. We rewrite $u=f_1+if_2$ with $f_1, f_2$ being real-valued functions. Then
	\[
	\blinear{u}{u}=\blinear{f_1}{f_1}+\blinear{f_2}{f_2}+i(\blinear{f_2}{f_1}-\blinear{f_1}{f_2}).
	\]
	By repeating a similar calculation as in the proof of \Cref{lemma:lemma1} for each of $\blinear{f_1}{f_1}$, $\blinear{f_2}{f_2}$, $\blinear{f_2}{f_1}$ and $\blinear{f_1}{f_2}$ and then adding, we obtain
	\begin{equation}\label{eq:expressions}
		\begin{split}
	&\blinear{u}{u}=\\
		&	
		\begin{cases}
			\cos(\frac{\beta-\alpha}{2}\pi)\left(\|\rdd{\frac{\beta-\alpha}{2}}\widetilde{	\rd{\alpha}f_1}\|^2_\mathbb{R}+\|\rdd{\frac{\beta-\alpha}{2}}\widetilde{	\rd{\alpha}f_2}\|^2_\mathbb{R}\right)+\\
			i\left(\int_{\mathbb{R}}\ldd{\frac{\beta-\alpha}{2}}\widetilde{	\rd{\alpha}f_1}\cdot\rdd{\frac{\beta-\alpha}{2}}\widetilde{	\rd{\alpha}f_2}\,dx-\int_{\mathbb{R}}\rdd{\frac{\beta-\alpha}{2}}\widetilde{	\rd{\alpha}f_1}\cdot\ldd{\frac{\beta-\alpha}{2}}\widetilde{\rd{\alpha}f_2}\,dx\right), & \beta>\alpha,\\
			\|\rd{\alpha}f_1\|^2_\Omega+\|\rd{\alpha}f_2\|^2_\Omega, &\beta=\alpha,\\
			\cos(\frac{\alpha-\beta}{2}\pi)	\left(\|\rdd{\frac{\alpha-\beta}{2}}\widetilde{\rd{\beta}f_1}\|^2_\mathbb{R}+\|\rdd{\frac{\alpha-\beta}{2}}\widetilde{\rd{\beta}f_2}\|^2_\mathbb{R}\right)+\\
			i\left(\int_{\mathbb{R}}\rdd{\frac{\alpha-\beta}{2}}\widetilde{\rd{\beta}f_1}\cdot\ldd{\frac{\alpha-\beta}{2}}\widetilde{\rd{\beta}f_2}\, dx-\int_{\mathbb{R}}\ldd{\frac{\alpha-\beta}{2}}\widetilde{\rd{\beta}f_1}\cdot\rdd{\frac{\alpha-\beta}{2}}\widetilde{\rd{\beta}f_2}\, dx  \right),
			& \beta<\alpha,
		\end{cases}
	\end{split}
	\end{equation}
	where the notation $\widetilde{\cdot}$ represents the zero extension as before.
	
	3. Recall \eqref{eq:argument}, it is clear from the above expresses in \eqref{eq:expressions} that, when $\beta=\alpha$, $\Arg \lambda =0$. In the next step, we only discuss $\Arg \lambda$ for the case $\beta>\alpha$ and the same conclusion can be obtained for the other case of $\beta<\alpha$.
	
	4. Let us estimate the imaginary part of $\blinear{u}{u}$, $\Im(\blinear{u}{u})$, for the case $\beta>\alpha$:
	\begin{equation}\label{eq:touseF}
		\int_{\mathbb{R}}\ldd{\frac{\beta-\alpha}{2}}\widetilde{	\rd{\alpha}f_1}\cdot\rdd{\frac{\beta-\alpha}{2}}\widetilde{	\rd{\alpha}f_2}\,dx-\int_{\mathbb{R}}\rdd{\frac{\beta-\alpha}{2}}\widetilde{	\rd{\alpha}f_1}\cdot\ldd{\frac{\beta-\alpha}{2}}\widetilde{\rd{\alpha}f_2}\,dx.
	\end{equation} 
Note with the aid of \Cref{lem:subtle-seminorm} that $\widetilde{	\rd{\alpha}f_1}, \widetilde{	\rd{\alpha}f_2}\in \hilbert{\frac{\beta-\alpha}{2}}$, hence, we apply the Plancherel's	formula to \eqref{eq:touseF} to arrive at
	\[
	\begin{aligned}
		&\big|\int_{\mathbb{R}}\ldd{\frac{\beta-\alpha}{2}}\widetilde{	\rd{\alpha}f_1}\cdot\rdd{\frac{\beta-\alpha}{2}}\widetilde{	\rd{\alpha}f_2}\,dx-\int_{\mathbb{R}}\rdd{\frac{\beta-\alpha}{2}}\widetilde{	\rd{\alpha}f_1}\cdot\ldd{\frac{\beta-\alpha}{2}}\widetilde{\rd{\alpha}f_2}\,dx\big |\\
		&=\big|\int_{\mathbb{R}}(2\pi i\xi)^{\beta-\alpha}\hat{g_2}\overline{\hat{g_1}}\,d\xi-\int_{\mathbb{R}}(-2\pi i\xi)^{\beta-\alpha}\hat{g_2}\overline{\hat{g_1}}\,d\xi\big |
	\end{aligned}
	\]
		where $g_1(x):=\widetilde{	\rd{\alpha}f_1}$, $g_2(x):= \widetilde{	\rd{\alpha}f_2}, x\in \mathbb{R}$, the Fourier transform is defined as  $\hat{f}(\xi):=\int_{\mathbb{R}} e^{-2\pi i x\xi}f(x) \,dx$ and  $(\mp i\xi)^{\sigma}$ is understood as the corresponding values of the main branch of the analytic function $z^{\sigma}$ in the complex plane with the cut along the positive half-axis, i.e. $(\mp i\xi)^{\sigma}=|\xi|^\sigma e^{\mp  \sigma \pi i \cdot \text{sign} (\xi)/2}$. 
Simplifying the terms leads to 
	\begin{equation}
		\begin{aligned}
			&=\big |\int_{\mathbb{R}}((2\pi i\xi)^{\beta-\alpha}-(-2\pi i\xi)^{\beta-\alpha})\hat{g_2}\overline{\hat{g_1}}\,d\xi\big |\\
			&=2\sin\frac{(\beta-\alpha)\pi}{2}\big|\int_{\mathbb{R}_+}|2\pi  \xi|^{\beta-\alpha}\hat{g_2}\overline{\hat{g_1}}\,d\xi-\int_{\mathbb{R}_-}|2\pi  \xi|^{\beta-\alpha}\hat{g_2}\overline{\hat{g_1}}\,d\xi\big|\\
			&\leq 2\sin\frac{(\beta-\alpha)\pi}{2}\int_{\mathbb{R}}||2\pi  \xi|^{\beta-\alpha}\hat{g_2}\overline{\hat{g_1}}| \, d\xi \\
			&=2\sin\frac{(\beta-\alpha)\pi}{2}\int_{\mathbb{R}}\widehat{\rdd{\frac{\beta-\alpha}{2}}g_2}\cdot\overline{\widehat{\rdd{\frac{\beta-\alpha}{2}}g_1} } \, d\xi 
		\end{aligned}
	\end{equation}
	by the second application of Plancherel's	formula and  H\"older inequality, this is
	\begin{equation}\label{eq:imaginarestimate}
		\begin{aligned}
			& \leq 2\sin\frac{(\beta-\alpha)\pi}{2}\cdot \|\rdd{\frac{\beta-\alpha}{2}}g_1\|_{L^2(\mathbb{R})}\cdot\|\rdd{\frac{\beta-\alpha}{2}}g_2\|_{L^2(\mathbb{R})}\\
			&\leq \sin\frac{(\beta-\alpha)\pi}{2} \left(\|\rdd{\frac{\beta-\alpha}{2}}g_1\|^2_{L^2(\mathbb{R})}+\|\rdd{\frac{\beta-\alpha}{2}}g_2\|^2_{L^2(\mathbb{R})}\right).
		\end{aligned}
	\end{equation}
By back substitution for  $g_1$ and $g_2$, we see
\[
=\sin\frac{(\beta-\alpha)\pi}{2} \left(\|\rdd{\frac{\beta-\alpha}{2}}\widetilde{	\rd{\alpha}f_1}\|^2_\mathbb{R}+\|\rdd{\frac{\beta-\alpha}{2}}\widetilde{	\rd{\alpha}f_2}\|^2_\mathbb{R}\right).
\]
Therefore, 
\[
\big|\Arg \lambda\big|=\big|\arctan \blinear{u}{u}\big|=\arctan \frac{\big|\Im(\blinear{u}{u})\big|}{\Re(\blinear{u}{u})} \leq \frac{(\beta-\alpha)\pi}{2}.
\]

5. Similarly, it can be checked  for the case $\beta<\alpha$ that
\[
\big|\Arg \lambda\big|\leq \frac{(\alpha-\beta)\pi}{2}.
\]
Combining all these three different cases, we thus have
\[
\big|\Arg \lambda\big|\leq \frac{|\alpha-\beta|\pi}{2},
\]
which closes the proof.
\end{proof}

\section{Principal eigenvalue}
In this section, we prove that there exists at least one real eigenvalue and the corresponding eigenfunction is strictly positive in $\Omega$. This is an analogy of the classic result of principal eigenvalue for nonsymmetric elliptic operators. The proof is constructive, following closely the original proof in Theorem 3, section 6.5.2, \cite{MR2597943}, which is based on the application of Schaefer's Fixed Point Theorem for convex sets, Hopf's Lemma and Maximum principle of fractional versions:
\begin{lemma}[Schaefer's Fixed Point Theorem]\label{the:schaefer}
	
	Let $T$ be a  continuous and compact mapping of a nonempty convex subset $K$ of a Banach space $X$ into itself, such that the set
	\[
	\{x\in K:x=\lambda Tx  \mbox{ for some } 0\leq \lambda \leq 1\}
	\]
	is bounded. Then $\displaystyle T$ has a fixed point in $K$.
\end{lemma}

\begin{lemma}[Maximum principle]\label{lem:maximum}
Suppose $u\in \hilbert{\frac{\alpha+\beta}{2}}$ is a true solution of	\begin{equation}\label{eq:intermidiateprob}
	\begin{cases}
		Au=f \quad\text{in} \quad \Omega,\quad f\in \LL,\\
		u(a)=u(b)=0.
	\end{cases}
\end{equation}
If $f\geq 0$ a.e., then $u(x)\geq0$ in $\Omega$; furthermore, if $f$ is a nonzero function, then the strict inequality holds, i.e., $u(x)>0$ in $\Omega$.
\end{lemma}

\begin{lemma}[Hopf's Lemma]\label{lemma:Hopf}
	Suppose $\li{\alpha}f\in C^1(\overline{\Omega})$ and $f$ being real-valued. Then the solution $u$ in \Cref{lem:maximum} is differentiable and satisfies
	\begin{equation}
		\lim_{x\to a^+}u'(x)=+\infty,\quad\lim_{x\to b^-}u'(x)=-\infty.
	\end{equation}
\end{lemma}

For a related discussion on Schaefer's Fixed Point Theorem, we refer to section 9.2.2, \cite{MR2597943}, while for more general statements of \Cref{lem:maximum} and \Cref{lemma:Hopf} and their proofs,  we refer to \cite{YLisubmitted}. Let us only mention that the crucial difference between classic and fractional-order Hopf's lemmas is not that directional derivatives at the boundary are finite, but rather that the derivatives blow up at the endpoints in fractional-order Hopf's lemma.  Now we state the main result.
\begin{theorem}\label{the:theorem-principal}
	There exists a positive real eigenvalue $\lambda$ for problem \eqref{eq:pde} and the corresponding eigenfunction is strictly positive in $\Omega$. 
\end{theorem}
\begin{proof}
1.  We consider real-valued functions throughout the proof. Let us define the linear operator $A^{-1}:\LL\rightarrow \hilbert{s}$, $s=\frac{\alpha+\beta}{2}$ as
		\begin{equation}\label{eq:operator}
		A^{-1}: f\rightarrow u, \quad f\in L^2(\Omega),
		\end{equation}
		where $u$ is the corresponding weak solution (therefore, true solution) of \eqref{eq:intermidiateprob}. Since $A^{-1}$ is bounded from $\LL$ to $\hilbert{s}$ by \eqref{eq:estimate} and $\hilbert{s}$ is compactly embedded into $\LL$, we have that $A^{-1}$: 
 $\LL \rightarrow \LL$ is a compact operator. On the other hand, if we  focus on the cone
\[
C:=\{u\in \hsss|u\geq 0 , x \in \Omega\},
\]
we further have that $A^{-1}$ maps $C$ into $ C$, according to maximum principle in \Cref{lem:maximum}. 

2.  We claim that
\begin{equation}\label{eq:claim}
\text{there exist $\aleph>0$ and a function $g\in \hsss$ such that  $A^{-1}g\geq \frac{1}{\aleph} g, x\in \Omega$}.
\end{equation}
In particular, we  choose a function $g\in C_0^\infty(\Omega)$  such that 
\[
g\in C, g\not\equiv 0.
\]
Then by applying the maximum principle and Hopf's lemma, we obtain
\[
\tilde{g}:=A^{-1}g>0 , x\in \Omega, \lim_{x\rightarrow a^+}D \tilde{g}=+\infty,\lim_{x\rightarrow b^-}D \tilde{g}=-\infty.
\]
This implies that there is a certain real constant $\aleph>0$ so that
\begin{equation}\label{eq:biggerthan}
 \tilde{g}\geq \frac{1}{\aleph} g, x\in \Omega,
\end{equation}
hence the claim holds. 

3. Hereafter, we pick and fix an arbitrary pair $(\aleph, g)$ satisfying the claim \eqref{eq:claim}.  We next assert that the set
\[
G_\epsilon:=\{u\in C|\text{there exists $0\leq t\leq \aleph+\epsilon$ such that $u=tA^{-1}[u+\epsilon g]$}\}
\]
is unbounded in $\hsss$ for any fixed $\epsilon>0$. For otherwise we could apply Schaefer's Fixed Point Theorem~\ref{the:schaefer} to deduce that for a certain $\epsilon>0$ the equation
\begin{equation}\label{eq:fixedpoint}
u=(\aleph+\epsilon)A^{-1}[u+\epsilon g]
\end{equation}
has a solution $u\in C$. However, this is impossible. To verify this, suppose~\eqref{eq:fixedpoint} holds,  using \eqref{eq:biggerthan} we compute
\begin{equation}\label{eq:uinequality}
u= (\aleph+\epsilon)A^{-1}[u]+(\aleph+\epsilon)A^{-1}[\epsilon g] \geq (\aleph+\epsilon)A^{-1}[\epsilon g]\geq \frac{\aleph+\epsilon}{\aleph}\epsilon g.
\end{equation}
Hence, \eqref{eq:uinequality} in turn yields
\[
u\geq(\aleph+\epsilon)A^{-1}[u]\geq \frac{(\aleph+\epsilon)^2}{\aleph}A^{-1}[\epsilon g]\geq \left(\frac{\aleph+\epsilon}{\aleph}\right)^2\epsilon g.
\]
Repeating this process, we inductively obtain
\[
u\geq \left(\frac{\aleph+\epsilon}{\aleph}\right)^n\epsilon g
\]
for any integer $n$. Since $\left(\frac{\aleph+\epsilon}{\aleph}\right)^n\rightarrow +\infty$ as $n\rightarrow +\infty$, it contradicts the boundedness of $u$ in $\Omega$.

4.  Owing to the unboundedness of $G_\epsilon$ and the arbitrariness of $\epsilon>0$,  we could let $\epsilon \rightarrow 0$. Then accordingly there exist $\{t_\epsilon\}$ and $\{v_\epsilon\}$ such that
\begin{equation}\label{eq:cha-1}
0\leq t_\epsilon\leq \aleph+\epsilon, v_\epsilon\in C \quad \text{and} \quad \|v_\epsilon\|_{\hsss}\rightarrow +\infty,
\end{equation}
satisfying
\[
v_\epsilon=t_\epsilon A^{-1}[v_\epsilon+\epsilon g].
\]
Therefore, we construct a bounded sequence $\{u_\epsilon\}$ by renormalizing
\[
u_\epsilon:=\frac{v_\epsilon}{\|v_\epsilon\|_{\hsss}}.
\]
Using the compactness of  intervals $[0, \aleph+\epsilon]$ and  the compactness of $A^{-1} $, we obtain convergent subsequences
\begin{equation}\label{eq:cha-2}
t_{\epsilon_n}\rightarrow t, \quad \text{and}\quad u_{\epsilon_n}\rightarrow u \quad \text{in}\quad \hsss,
\end{equation}
satisfying
\[
u_{\epsilon_n}=t_{\epsilon_n} A^{-1}[u_{\epsilon_n}+\epsilon_n \frac{g}{\|v_\epsilon\|_{\hsss}}].
\]
We deduce in limit that $u=tA^{-1}u$, which reads
	\begin{equation}
\begin{cases}
Au=tu \quad\text{in} \quad \Omega,\\
u(a)=u(b)=0.
\end{cases}
\end{equation}
Clearly, $0<t\leq \aleph$, we see the desired results by letting $\lambda=t$. Apparently, by our construction, the eigenfunction $u$ is strictly positive in $\Omega$.
\end{proof}

Herein, $\lambda$ may or may not be the smallest one among all the (real and complex) eigenvalues in absolute value, remaining as an open question.  It is not even clear whether there exists a complex eigenvalue with non-zero imaginary part. Furthermore, the simplicity of the eigenvalues also remains unknown.

\section{A numerical study}
This section is devoted to the study of numerical spectral approximation of the nonsymmetric elliptic operator defined in \eqref{eq:pde}. The variational formulation developed in \Cref{sec:variation} provides a direct way of numerical computation for approximating the eigenpairs. 
In particular, we adopt the widely-used finite element method (FEM) to solve the problem \eqref{eq:pde}. 
The method reaches a mature stage and we refer to, for example, \cite{c2002,ern2004theory,br2008} for its development. 
In this section, we first present the linear FEM for problem \eqref{eq:pde} to reach the matrix eigenvalue problem.
We then show various numerical examples to demonstrate the spectral properties of the fractional diffusion operator.

\subsection{The linear finite element method}

The variational formulation of~\eqref{eq:pde} at the continuous level is to find the eigenpair $(\lambda \in \mathbb{C}, u \in \hsss) $ with $\|u\|_\Omega=1$ such that
\begin{equation} \label{eq:vf}
B[u, w] =  \lambda (u, w)_\Omega, \quad \forall \ w \in\hsss, 
\end{equation}
where the sesquilinear form is defined in \eqref{eq:bilinear}.
The eigenvalue problem~\eqref{eq:vf} is equivalent to the original problem~\eqref{eq:pde}.  
They have a countable set of eigenvalues and an associated set of eigenfunctions $\{ u_j\}_{j=1}^\infty$. 
When $\alpha=\beta$, the bilinear form is symmetric and both matrices $K$ and $M$ are symmetric. 
In this case, the eigenfunctions can be orthnormalized, 
meaning $ (u_j, u_k)_\Omega = \delta_{jk}, $ where $\delta_{jk} =1$ is the Kronecker delta.  
In our numerical study, we always sort the eigenvalues accordingly to their magnitudes in ascending order.
Each eigenvalue is paired with the corresponding eigenfunctions.

At the discrete level, we first partition the domain $\Omega = [0,1]$ into a grid with nodes $\{0=x_0, x_1, x_2, \cdots, x_N = 1\}$.
Let $\tau_j = [x_{j-1}, x_j]$ and $\mathcal{T}_h$ denote a general element and its collection, respectively, such that $\overline \Omega = \cup_{\tau \in \mathcal{T}_h} \tau$.  
Also, let $h_\tau = \text{diameter}(\tau)$ and $h = \max_{\tau \in \mathcal{T}_h} h_\tau$.  
For simplicity, we use a mesh with $N$ uniform elements.  
Thus, $h=1/N$ and $x_j = jh, j=0,1,\cdots, N$.
With this in mind, the linear FEM space is 
\begin{equation}
V^h = \{ w \in C^0(\Omega): w|_\tau \in \mathbb{P}_1(\tau), w(0)=w(1)= 0, \forall \tau \in \mathcal{T}_h \},
\end{equation}
where $\mathbb{P}_1$ is a space of linear polynomials. 
The linear FEM of~\eqref{eq:pde} or equivalently~\eqref{eq:vf} seeks an eigenvalue $\lambda^h$ and its associated eigenfunction $u^h \in V^h$ with $\| u^h \|_\Omega = 1$ such that
\begin{equation} \label{eq:vfh}
B[u^h, w^h] =  \lambda^h (u^h, w^h)_\Omega, \quad \forall \ w^h \in V^h.
\end{equation} 

The linear FEM space is characterized by a set of basis functions which are the so-called ``hat" functions $\phi_j, j=1,2,\cdots, N-1$ (the first one when $j=0$ and the last one when $j=N$ are removed to satisfy the homogeneous boundary condition). Thus, $V^h = \{ \phi_j \}_{j=1}^{N-1}$.
With this characterization, an eigenfunction can be written as a linear combination of the basis functions, i.e.,
\begin{equation}
u^h =  \sum_{j=1}^{N-1} U_j \phi_j.
\end{equation}
At the algebraic level, substituting all the basis functions for $w^h$ in~\eqref{eq:vfh} leads to the generalized matrix eigenvalue problem (GMEVP)
 \begin{equation} \label{eq:mevp}
  KU = \lambda^h MU,
\end{equation}
where $K_{kl} = B[\phi_l, \phi_k], M_{kl} =  (\phi_l, \phi_k)_\Omega,$ and $U$ is the eigenvector representing the coefficients of the basis functions.  
In general, the matrix $K$ is dense while the mass matrix $M$ is sparse.
Once the matrix eigenvalue problem~\eqref{eq:mevp} is solved, we sort the eigenpairs in ascending order.

\begin{figure}[h!]
\centering
\includegraphics[height=8cm]{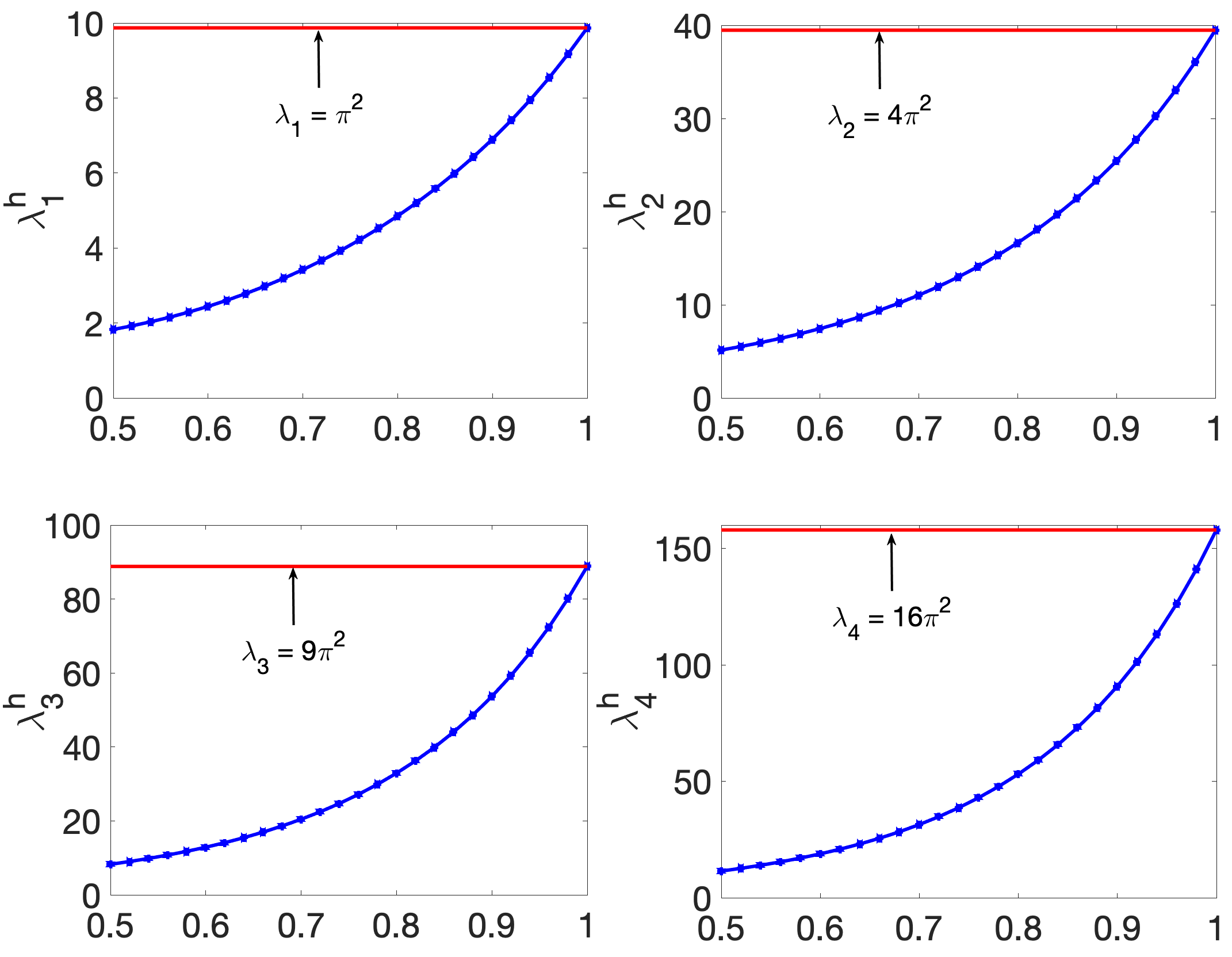} 
\caption{The first four eigenvalues $\lambda_j^h, j=1,2,3,4$ of \eqref{eq:pde} with respect to $\alpha = \beta \in [0.5, 1]$.}
\label{fig:ev1234}
\end{figure}

\begin{figure}[h!]
\centering
\includegraphics[height=8cm]{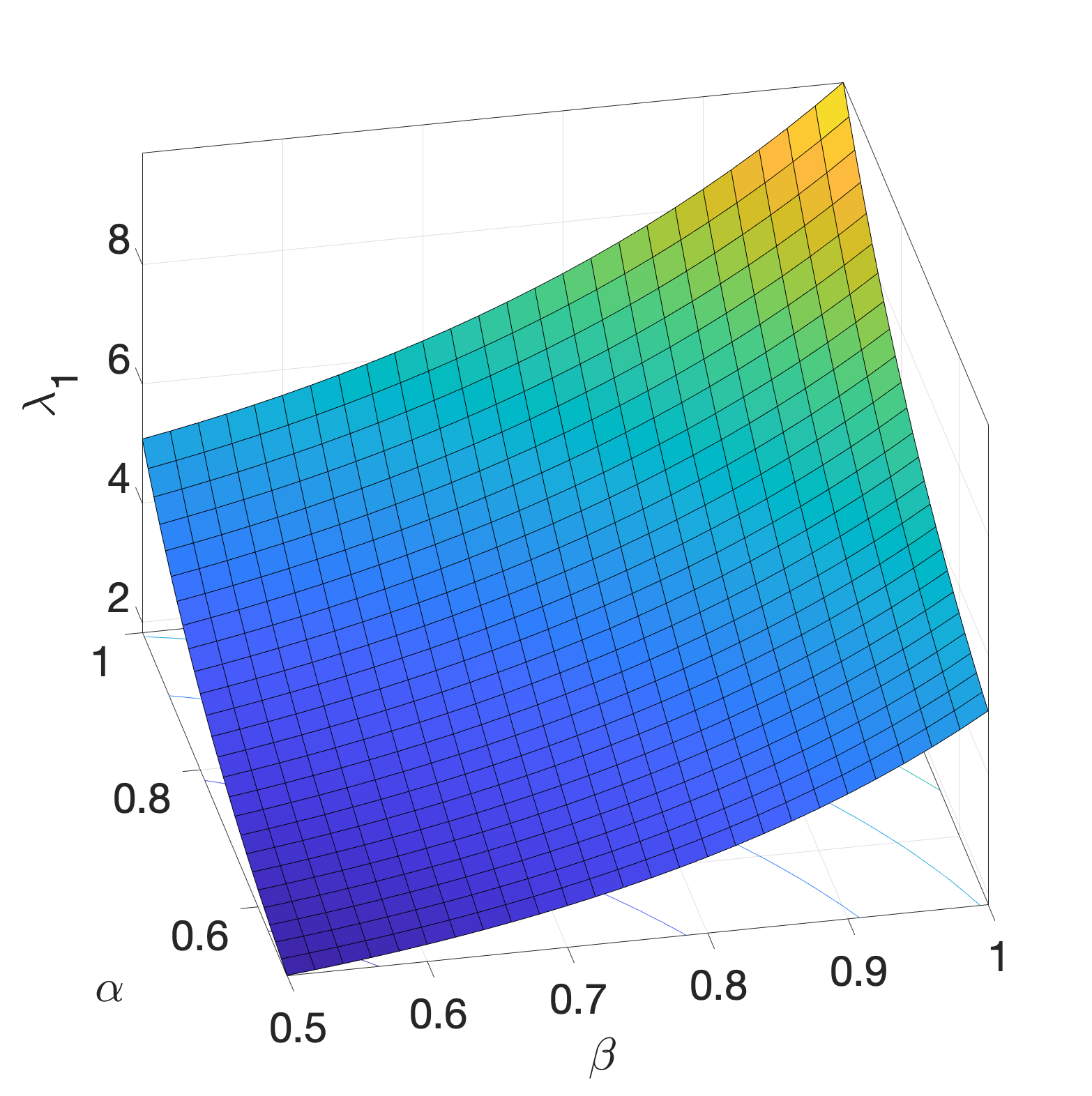} 
\vspace{-0.6cm}
\caption{The first eigenvalue $\lambda_1$ of \eqref{eq:pde} with respect to $(\alpha, \beta) \in [0.5, 1] \times [0.5, 1]$ with a grid of size $25 \times 25$.}
\label{fig:evs}
\end{figure}

\begin{figure}[h!]
\centering
\includegraphics[height=8cm]{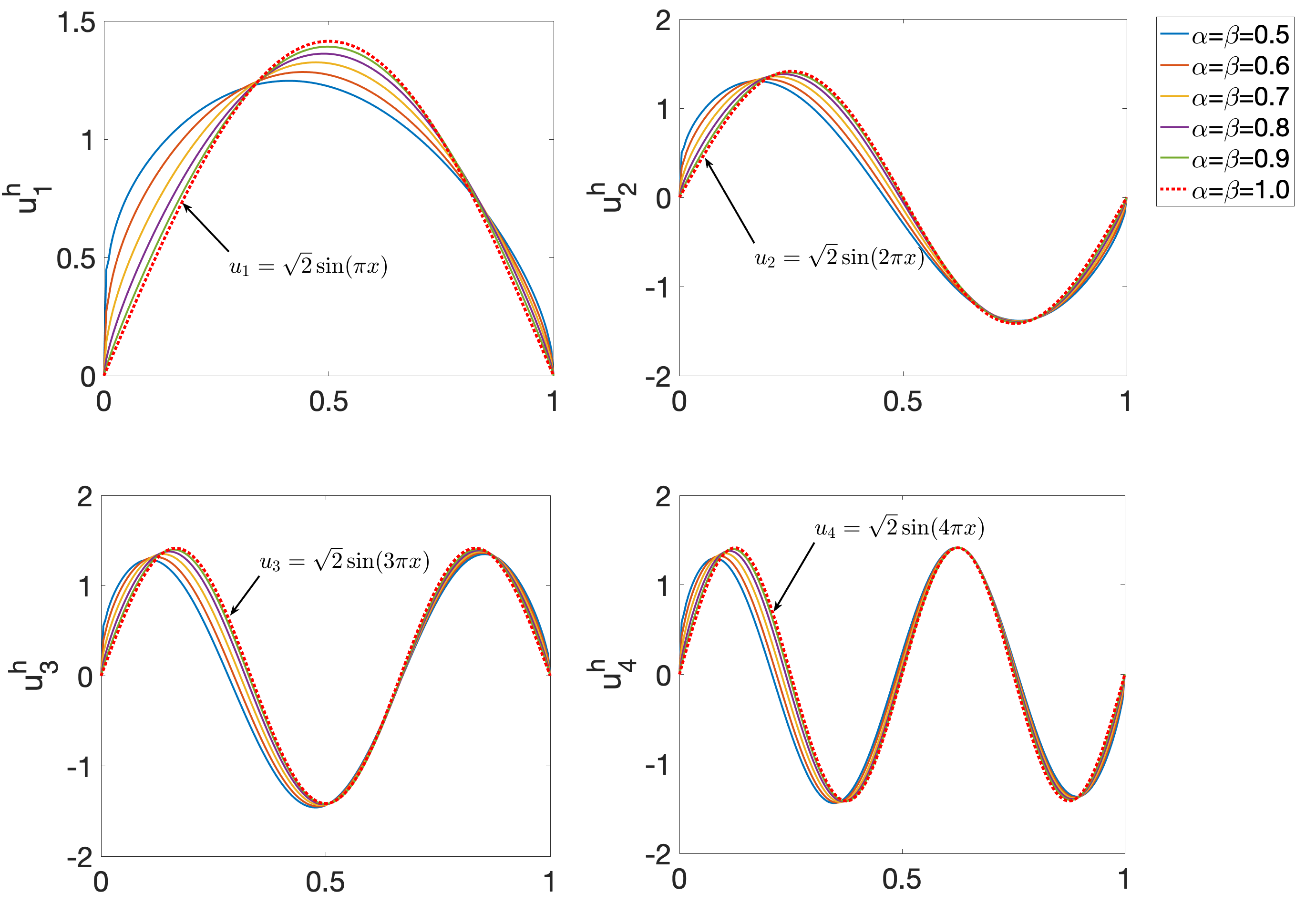} 
\caption{The first four eigenfunctions $u_j^h, j=1,2,3,4$ of \eqref{eq:pde} with respect to $\alpha = \beta \in [0.5, 1]$.}
\label{fig:efs}
\end{figure}

\subsection{The Laplace operator as a limiting case when $\alpha, \beta \to 1$}

When $\alpha = \beta = 1$, the  problem \eqref{eq:pde} reduces to the Laplacian eigenvalue problem,
which has analytical eigenpairs 
$$
(\lambda_j, u_j) = (j^2 \pi^2, \sqrt{2} \sin(j\pi x)), \qquad j =1,2,\cdots.
$$
Figure \ref{fig:ev1234} shows that when $\alpha= \beta \to 1$, the first four eigenvalues $\lambda_j \to j^2\pi^2, j=1,2,3,4$.
Herein, we use linear FEM with $200$ uniform elements. 
Figure \ref{fig:evs} shows the first eigenvalue $\lambda_1$ with respect to various values of $\alpha$ and $\beta$.
When $\alpha=\beta$, the numerical eigenvalues are real and positive.
Also, the first eigenvalue is real and positive for any values of  $\alpha$ and $\beta$.
For eigenfunctions, Figure \ref{fig:efs} shows that when $\alpha= \beta \to 1$, the first four eigenfunctions $u_j^h \to u_j = \sqrt{2} \sin(j\pi x), j=1,2,3,4$.
The first eigenfunction is always real.
This numerical spectral observation confirms our expectation that the fractional diffusion operator $A$ defined in \eqref{eq:pde} converges to the Laplacian  as $\alpha, \beta \to 1$.

\subsection{The distribution of eigenvalues}

We now take a closer look at the distribution of eigenvalues. 
Since linear FEM is a numerical method for solving the problem \eqref{eq:pde}, it has numerical errors. 
We use a fine mesh with $N=400$ elements to help reduce the errors.
The errors increase from small to large eigenvalues. 
We thus equally divide the overall spectra into three regions: numerically accurate region (blue), transitional region (green), and numerically inaccurate region (red).

\begin{figure}[h!]
\centering
\includegraphics[height=8cm]{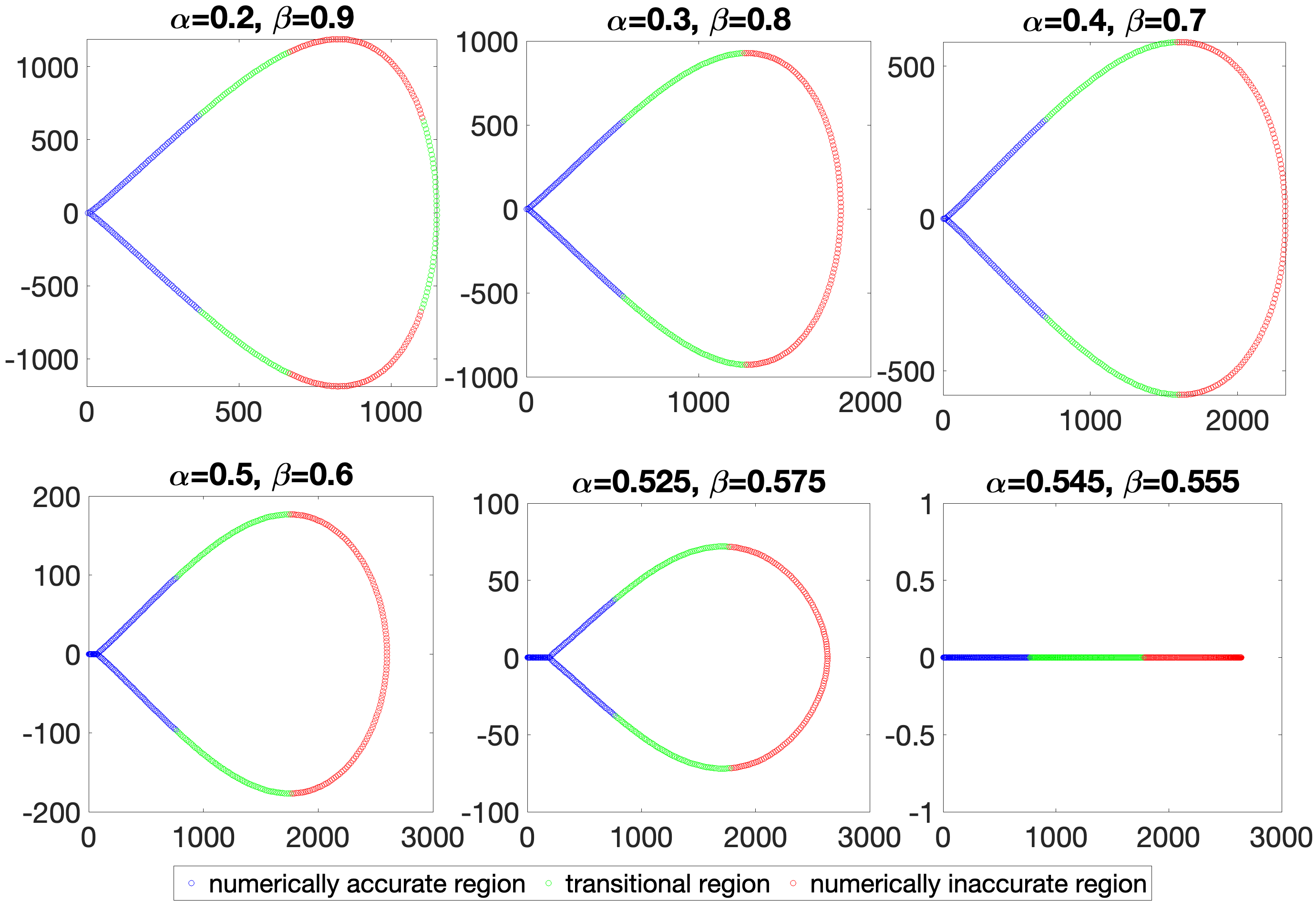} 
\caption{The distribution of eigenvalues for different values of $\alpha$ and $\beta$ such that $\alpha + \beta = 1.1.$ }
\label{fig:evdis1}
\end{figure}

\begin{figure}[h!]
\centering
\includegraphics[height=8cm]{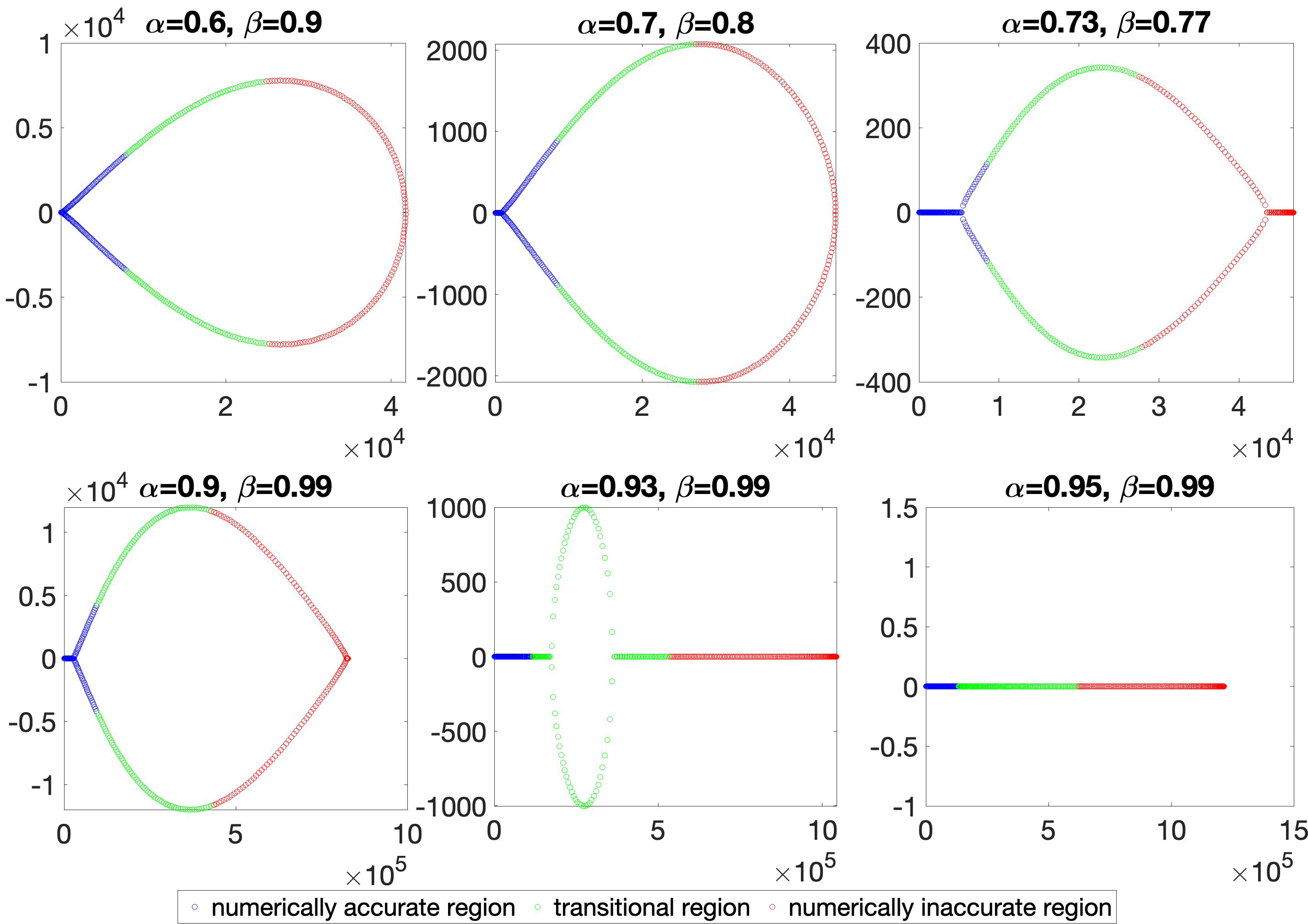} 
\caption{The distribution of eigenvalues for different values of $\alpha$ and $\beta$ such that $\alpha + \beta = 1.5$ (top plots) and $\beta=0.99$ (bottom plots). }
\label{fig:evdis2}
\end{figure}

Figure \ref{fig:evdis1} shows the distribution of eigenvalues when $\alpha + \beta = 1.1$ with different values of $\alpha$ and $\beta$. 
Figure \ref{fig:evdis2} shows the cases when $\alpha + \beta = 1.5$ and   $\beta$ is fixed.
The horizontal axis represents the real part while the vertical axis represents the imaginary part.

The first eigenvalues are positive and all real parts of the eigenvalues are positive as well. 
When $| \alpha- \beta| \to 0$, we can observe that more real and positive eigenvalues appear in the spectra. 
In particular, Figure \ref{fig:evdis3} shows the first 100 eigenvalues (all numerically accurate) when using different values of $\alpha$ and $\beta$ for each fixed sum $\alpha + \beta = 1.1, 1.2, 1.4, 1.6, 1.8, 1.96$.
 
For the case $\alpha + \beta = 1.2$, Figure \ref{fig:evdis4} shows the bounds of the angles given by $\frac{|\alpha -\beta| \pi}{2}$.
We observe that all the eigenvalues are bounded by these angles, which confirms our Theorem \ref{theorem:range}.

\begin{figure}[h!]
\centering
\includegraphics[height=7cm]{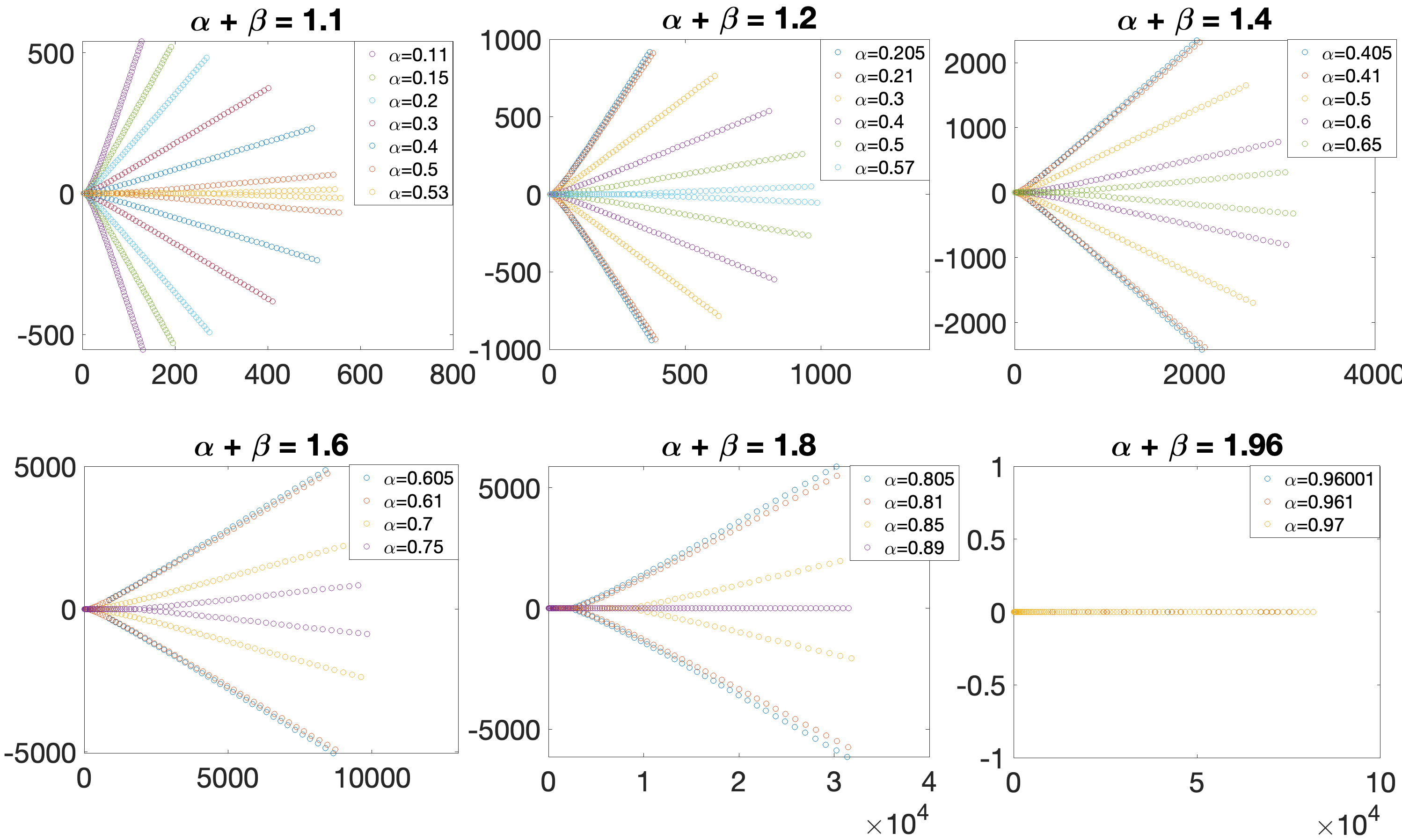} 
\caption{The distribution of the first 100 eigenvalues for different values of $\alpha$ and $\beta$ such that $\alpha + \beta = 1.1, 1.2, 1.4, 1.6, 1.8, 1.96$. }
\label{fig:evdis3}
\end{figure}

\begin{figure}[h!]
\centering
\includegraphics[height=7cm]{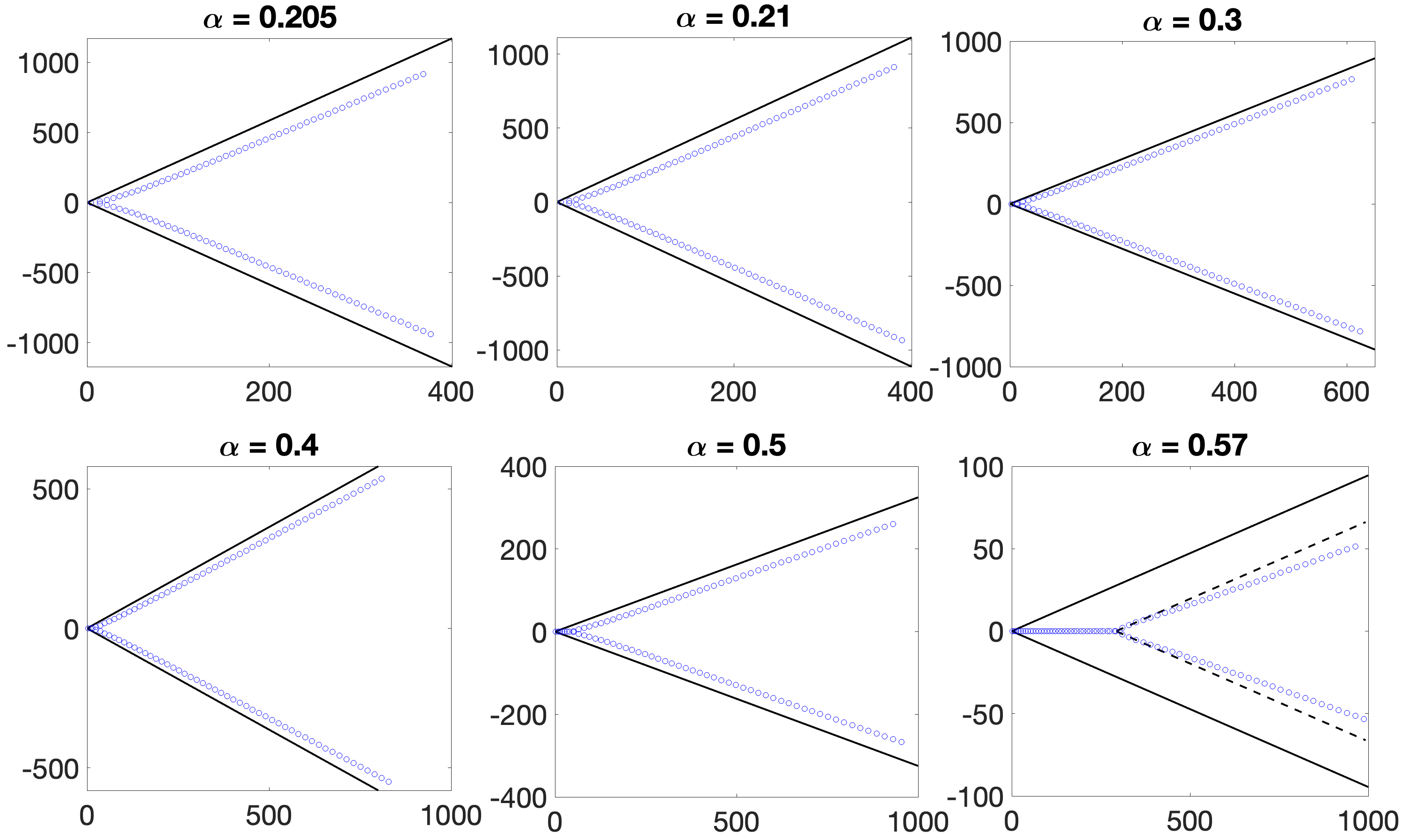} 
\caption{The distribution of the first 100 eigenvalues for different values of $\alpha$ and $\beta$ such that $\alpha + \beta = 1.2$. The solid black lines represent the angle given by $\frac{|\alpha -\beta| \pi}{2}$. The dotted black lines in the last plot are a horizontal right shift with a scale of 292.45 (the largest real eigenvalue before complex branches) of the solid black lines.}
\label{fig:evdis4}
\end{figure}

\subsection{The complex eigenfunctions}

To study the eigenfunctions, we fix the mesh with $N=200$ uniform elements. 
Figures \ref{fig:ef1}--\ref{fig:ef4} show the eigenfunctions $u_j^h, j=1,2,3,11,21,31$ for different values of $\alpha$.
Therein, we fix $\beta=0.9.$ 
For the case when $\alpha=0.2$, only the first eigenfunction is real and
all other eigenfunctions are complex. 
As $\alpha$ increases, we start seeing more and more real eigenfunctions. 
The number of real eigenfunctions with respect to $\alpha$ and $\beta$ is shown in Figure \ref{fig:nev}.
When the eigenfunction is real, its corresponding eigenvalue is also real.
Although, for any $1<\alpha + \beta \le 2$, it can be observed numerically that there exists at least one real eigenpair, an instructive observation is that more and more real eigenpairs appear (however, always finite) when either the sum $\alpha + \beta$ increases or $\alpha$ and $\beta$ are getting closer, i.e. $|\alpha-\beta| \to 0$.

\begin{figure}[h!]
\hspace{-1cm}
\includegraphics[height=7cm]{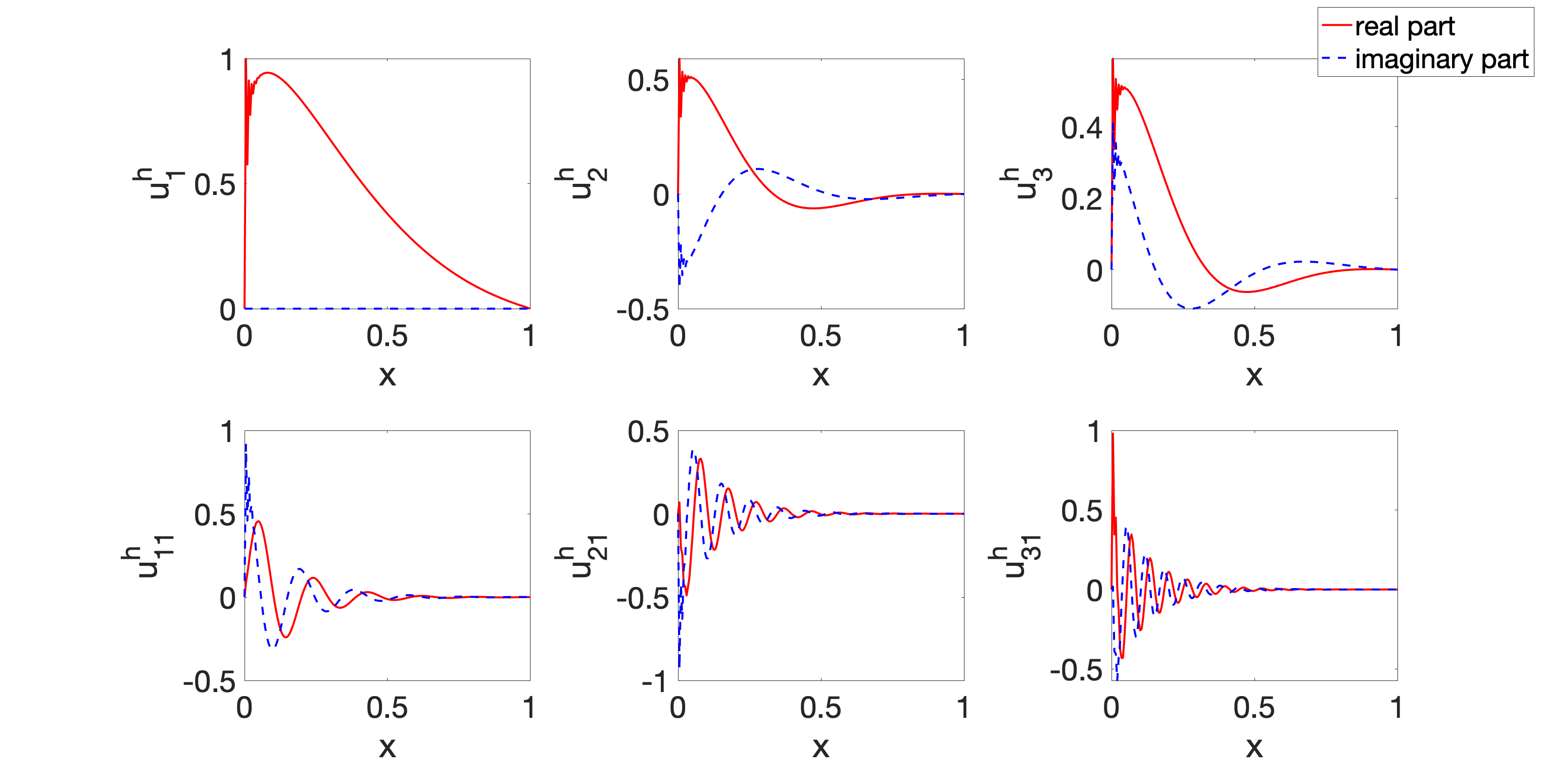} 
\caption{Several eigenfunctions when using linear FEM with 200 uniform elements with $\alpha=0.2$ and $\beta=0.9$. }
\label{fig:ef1}
\end{figure}

\begin{figure}[h!]
\hspace{-1cm}
\includegraphics[height=7cm]{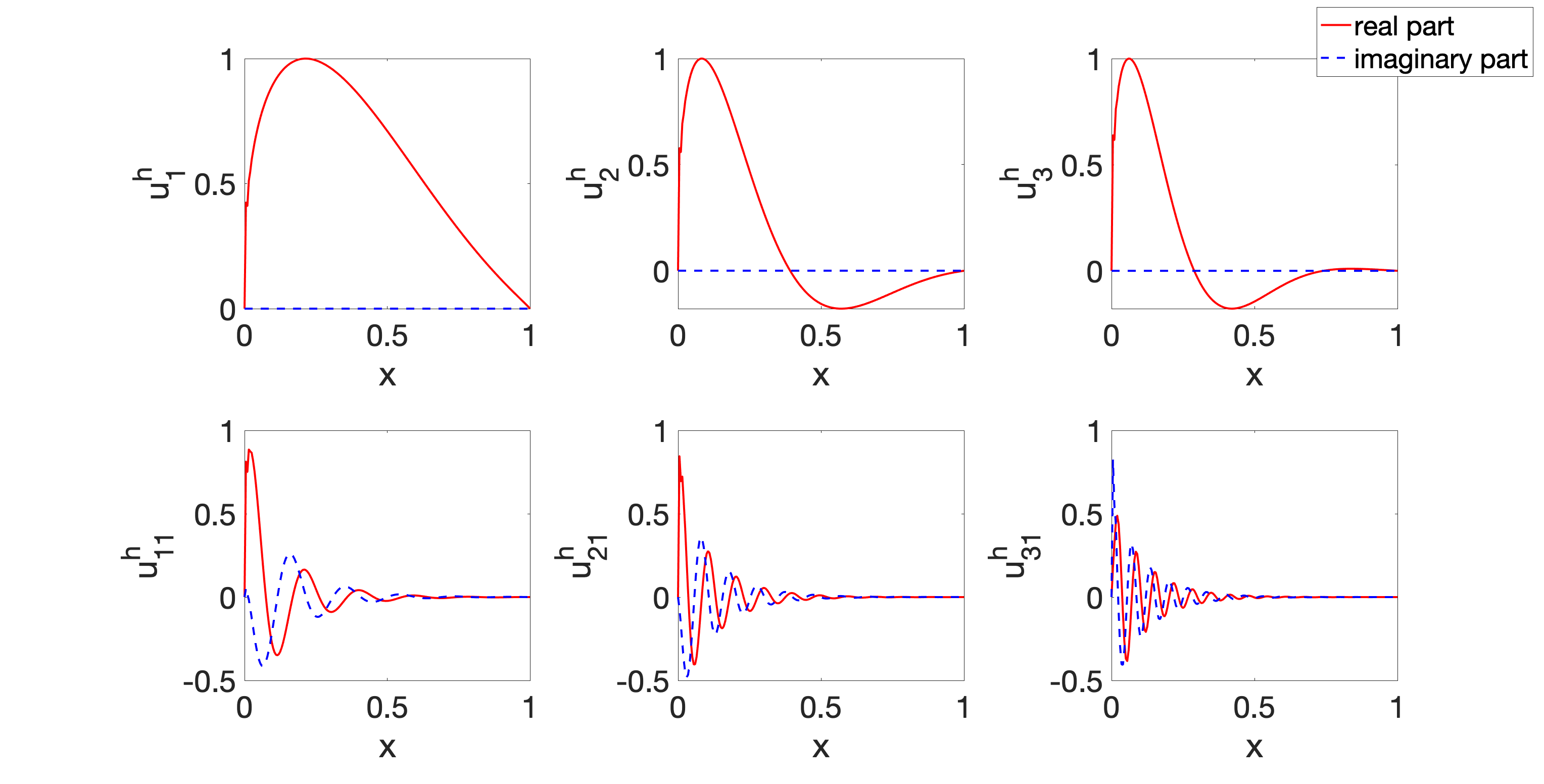} 
\caption{Several eigenfunctions when using linear FEM with 200 uniform elements with $\alpha=0.4$ and $\beta=0.9$. }
\label{fig:ef2}
\end{figure}

\begin{figure}[h!]
\hspace{-1cm}
\includegraphics[height=7cm]{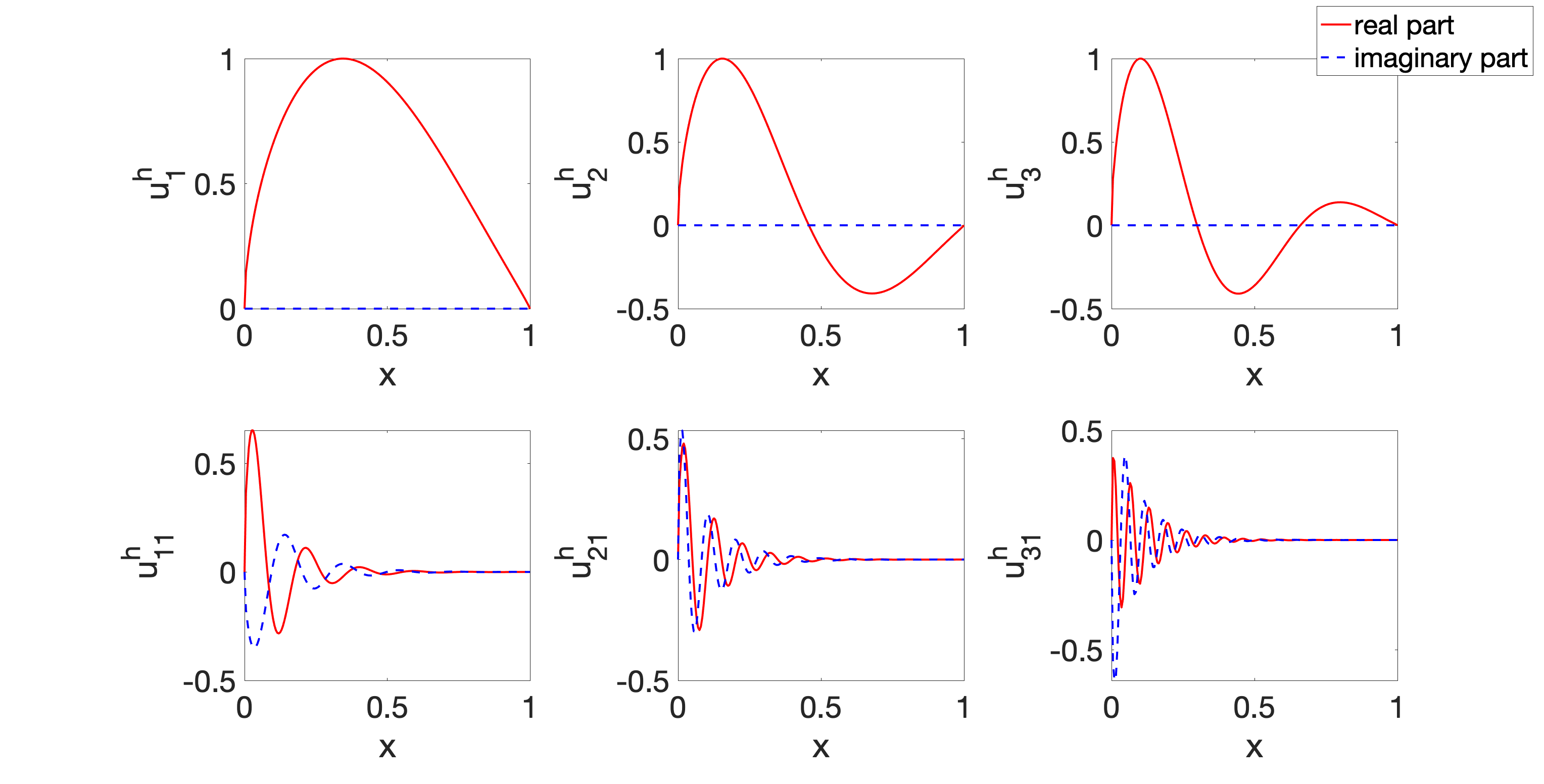} 
\caption{Several eigenfunctions when using linear FEM with 200 uniform elements with $\alpha=0.6$ and $\beta=0.9$. }
\label{fig:ef3}
\end{figure}

\begin{figure}[h!]
\hspace{-1cm}
\includegraphics[height=7cm]{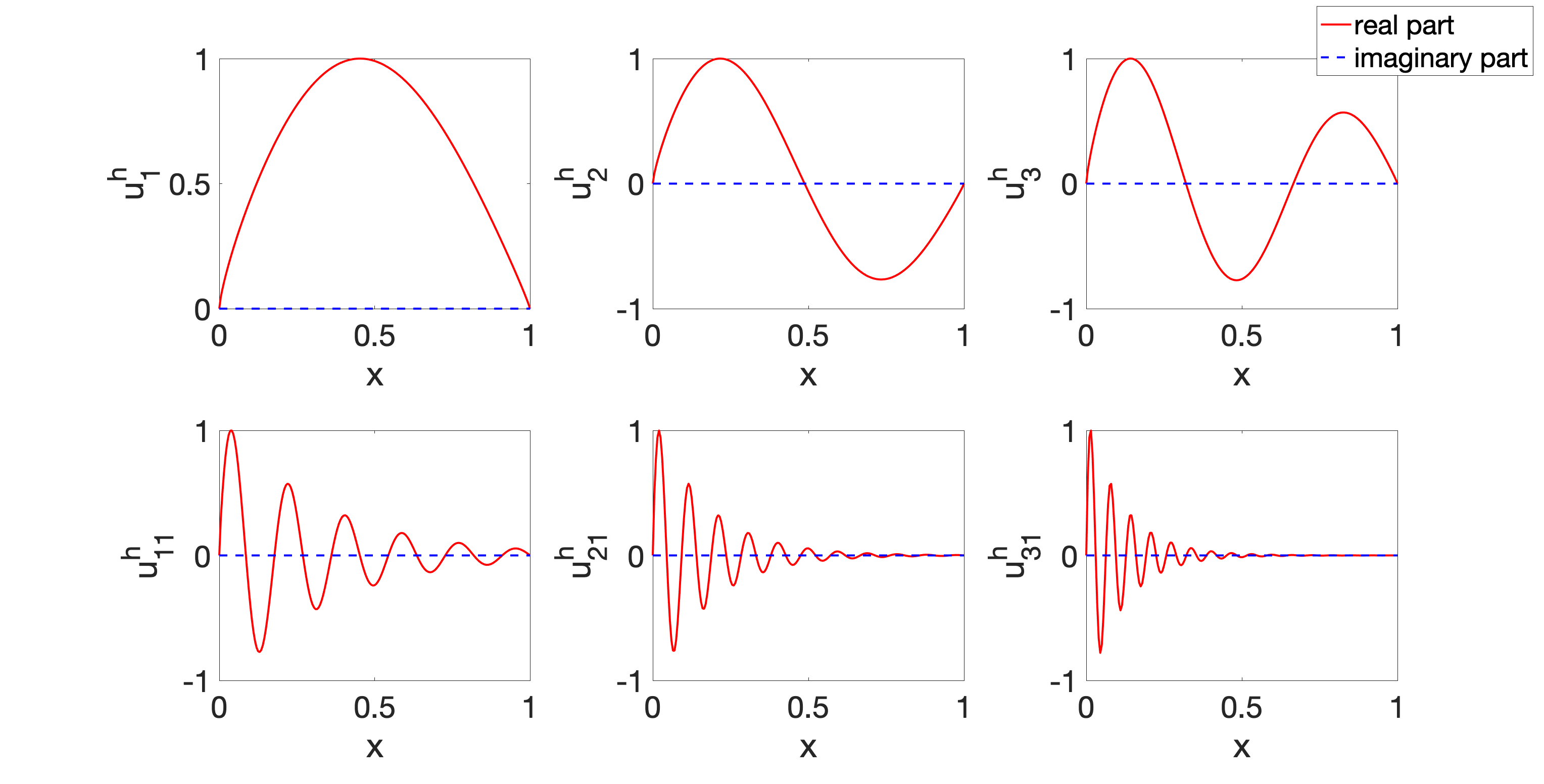} 
\caption{Several eigenfunctions when using linear FEM with 200 uniform elements with $\alpha=0.8$ and $\beta=0.9$. }
\label{fig:ef4}
\end{figure}

\begin{figure}[h!]
\centering
\includegraphics[height=7cm]{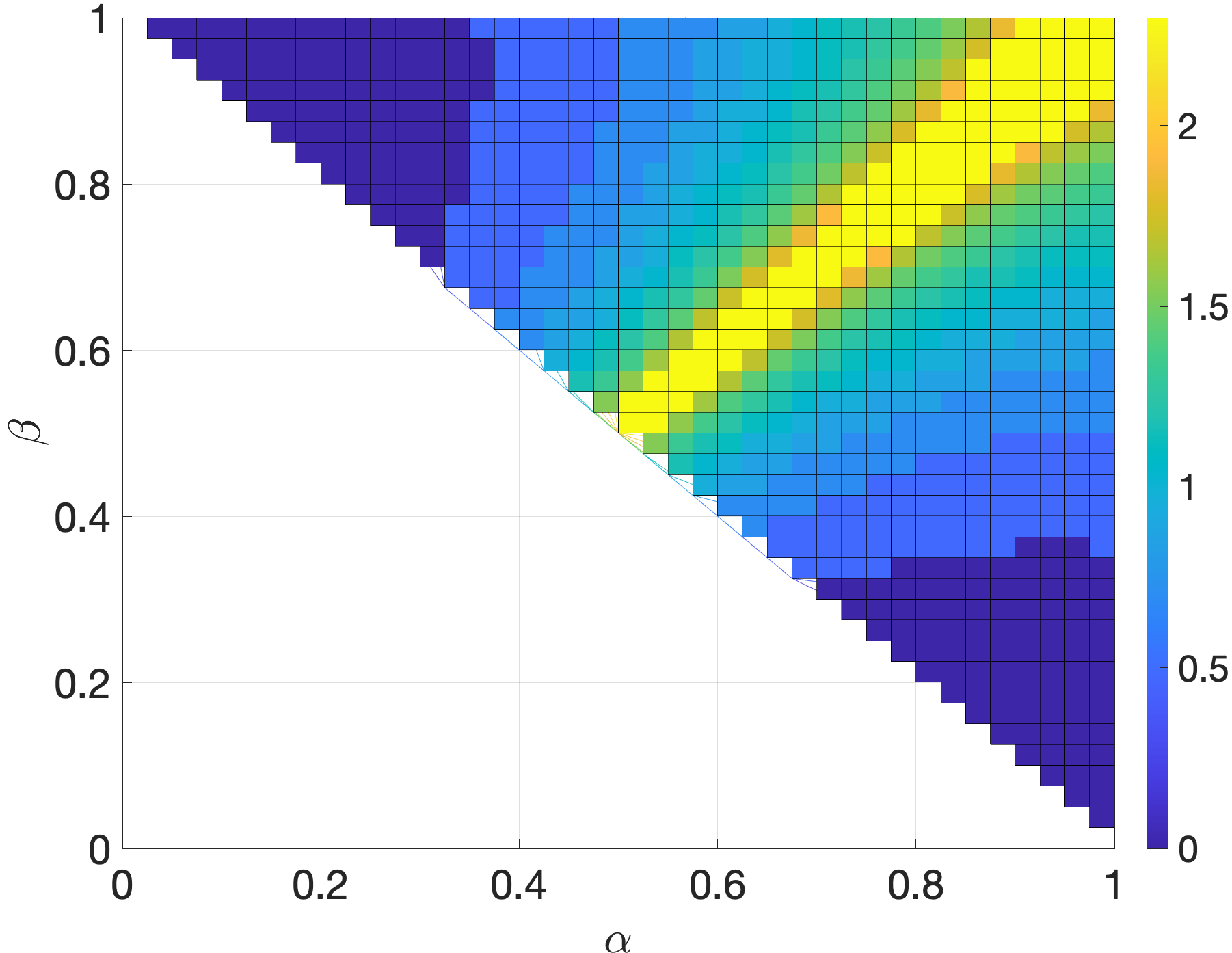} 
\caption{Number (in common logarithmic scale) of real eigenfunctions (or eigenvalues) with respect to $(\alpha, \beta) \in [0, 1] \times [0, 1]$ with $1< \alpha + \beta \le 2$ and a grid of size $40 \times 40.$ }
\label{fig:nev}
\end{figure}

%
%
%
%
%
%
%
%
%

\section{Conclusion}
In this work, we study various spectral properties of problem \eqref{eq:pde} involving both left- and right-sided Riemann-Liouville derivatives. This fractional diffusion operator is nonlocal and non-symmetric, which brings a challenge for doing analysis and requires developing new tools and framework. We find that all the possible eigenvalues are located at the core on the complex plane given by $|\Arg \lambda|\leq \frac{|\beta-\alpha|\pi }{2}$, and we also prove that there exists at least one real eigenvalue. All these results on the distribution of eigenvalues are confirmed in our numerical experiments. However, many other important questions remain open such as the simplicity of eigenvalues, the existence of complex eigenvalues with non-zero imaginary parts (numerically observed, but theoretical justification is missing), the finitude of real eigenfunctions, etc.

 Combing our current results and other elliptic results established in previous work \cite{YLisubmitted}, we see that the classic elliptic theory has its counterparts for problem \eqref{eq:pde}, however, exhibiting new interesting features and differences in the distribution of eigenvalues and the graphs of eigenfunctions. By constructing the suitable sesquilinear form, we can conduct numerical experiments with the finite element method, visualize the distribution of eigenvalues, and plot eigenfunctions from different angles. The numerical analysis supports our analytical results and, more importantly, suggests that there are only finitely many real eigenvalues except the symmetric case $\alpha=\beta$ and the number of real eigenpairs increase as either $\alpha+\beta$ increases or $\alpha$ and $\beta$ are getting closer (or both). These instructive numerical observations provide insights into the asymptotic behavior of the spectral problem \eqref{eq:pde} and deserve further theoretical exploration in the future.

\appendix 
\section{Definitions of fractional Riemann-Liouville integrals and derivatives and Poincar\'e-Friedrichs inequality}\label{appendix-sobolev}

We follow \cite{MR1347689} and list the definition of the fractional Riemann-Liouville integrals and derivatives as below. 

\begin{definition}\label{def:lrintegralderivative}
	Let $t\geq 0$ and  $n = \lceil t \rceil$ be the smallest integer such that $n> t$.  The left-sided Riemann-Liouville integral and derivative  are formally defined as
	\begin{equation*}
		(\li{t}u)(x):=\frac{1}{\Gamma(t)}\int_{a}^{x}\frac{u(t)dt}{(x-t)^{1-t}},\,\,\, 
		(\ld{t}u)(x):=\left(\frac{d}{dx}\right)^n \li{n-t}u,\quad x>a.
	\end{equation*}
Similarly, the right-sided Riemann-Liouville integral and derivative are defined as
	\begin{equation*}
		(\ri{t}u)(x):=\frac{1}{\Gamma(t)}\int_{x}^{b}\frac{u(t)dt}{(t-x)^{1-t}},\quad
		(\rd{t}u)(x):=\left(-\frac{d}{dx}\right)^n \ri{n-t}u,\quad x<b.
	\end{equation*}
 For the sake of notational convenience, if $a=-\infty$, we write 
	\[
	\lii{t}:={\mathcal{I}^t_{-\infty+}},\quad \ldd{t}:={\mathcal{D}^t_{-\infty+}},
	\]
	respectively; and if $b=+\infty$, write
	
\[
\rii{t}:={\mathcal{I}^t_{+\infty-}},\quad \rdd{t}:={\mathcal{D}^t_{+\infty-}},
\] 
 respectively.
\end{definition}

%
%

%
%
\begin{property}(Fractional  Poincar\'e-Friedrichs)\label{poincareinequality}
 Given $s>0$ and bounded interval $\Omega=(a,b)$, there holds
\[
\|u\|_{\LL}\leq \frac{(b-a)^s}{s\Gamma(s)} |u|_{\hsss}
\]
for any $u\in \hsss$.
\end{property}
\begin{proof}
	Given  $u\in \hsss$, by \Cref{pro:alternate-form}, $u$ can be represented as
	\[
\ld{s}u\in \LL	\quad \text{and}\quad u(x)=\li{s}\ld{s}u.
	\]
	By the boundedness of $\li{s}$ (eq. (2.72),  p. 48, \cite{MR1347689}), we have
	\[
	\|u\|_{\LL}\leq \frac{(b-a)^s}{s\Gamma(s)} \|\ld{s}u\|_{\LL} \leq 
	\frac{(b-a)^s}{s\Gamma(s)}|u|_{\hsss}.
	\]
\end{proof}

\bibliography{mybibfile}
\end{document}